\documentclass[11pt,a4paper]{amsart}

\usepackage{amssymb, amsthm, amsmath, amsfonts, mathrsfs, dsfont, stmaryrd}
\usepackage{mathtools}

\usepackage{tikz-cd}
\usetikzlibrary{
    calc, 3d, patterns, arrows.meta, decorations.pathmorphing, 
    decorations.pathreplacing, through, shapes, matrix
}

\tikzset{
    labl/.style={anchor=south, rotate=-90, inner sep=.6mm},
    symbol/.style={draw=none, every to/.append style={edge node={node [sloped, allow upside down, auto=false]{$#1$}}}},
    ext/.style={circle, draw, inner sep=1pt},
    int/.style={circle, draw, fill, inner sep=1pt},
    exte/.style={circle, draw, inner sep=3pt},
    inte/.style={circle, draw, fill, inner sep=3pt},
    diagram/.style={matrix of math nodes, row sep=3em, column sep=2.5em, text height=1.5ex, text depth=0.25ex},
    diagram2/.style={matrix of math nodes, row sep=0.5em, column sep=0.5em, text height=1.5ex, text depth=0.25ex},
    every picture/.style={baseline=-.65ex},
    every loop/.style={draw},
    rloop/.style={out=10, in=-10, loop, distance=3em},
    aloop/.style={out=100, in=80, loop, distance=3em}
}

\usepackage[left=1.5cm, right=1.5cm, top=3cm, bottom=3cm]{geometry}
\usepackage{indentfirst}
\usepackage{graphicx}
\usepackage{wrapfig, scalerel}
\usepackage{enumitem}

\usepackage{color}
\definecolor{darkblue}{rgb}{0,0.1,.5}
\usepackage[colorlinks=true, linkcolor=darkblue, urlcolor=darkblue, citecolor=darkblue]{hyperref}

\usepackage[labelformat=empty]{caption}

\makeatletter
\newcommand*{\RN}[1]{\expandafter\@slowromancap\romannumeral #1@}
\makeatother
\newcommand{\sbullet}[1][.5]{\mathbin{\vcenter{\hbox{\scalebox{#1}{$\bullet$}}}}}
\newcommand{\hr}[2][]{\hyperref[#2]{#1~\ref{#2}}}
\newcommand*\rcolon{%
        \nobreak
        \mskip6mu plus1mu
        \mathpunct{}%
        \nonscript
        \mkern-\thinmuskip
        {:}%
        \mskip2mu
        \relax
}

\def\R{\mathbb{R}}

\def\Z{\mathbb{Z}}
\def\Q{\mathbb{Q}}
\def\cone{\mathop{\mathrm{Cone}}\nolimits}

\def\emb{\mathop{\mathrm{Emb}}\nolimits}
\def\semb{\mathop{\mathrm{sEmb}}\nolimits}
\def\imm{\mathop{\mathrm{Imm}}\nolimits}
\def\Frm{\mathop{\mathrm{Fr}}^m\nolimits}
\def\Path{\mathop{\mathrm{Path}}\nolimits}
\def\map{\mathop{\mathrm{Map}}\nolimits}
\def\mor{\mathop{\mathrm{Mor}}\nolimits}

\def\conf{\mathop{\mathrm{Conf}}\nolimits}
\def\hofib{\mathop{\mathrm{hofib}}\nolimits}
\def\res{\mathop{\mathrm{Res}}\nolimits}
\def\cores{\mathop{\mathrm{coRes}}\nolimits}
\def\ind{\mathop{\mathrm{Ind}}\nolimits}
\def\coind{\mathop{\mathrm{coInd}}\nolimits}
\def\mc{\mathop{\mathsf{MC}}_{\sbullet}\nolimits}
\def\hgc{\mathop{\mathrm{HGC}}\nolimits}
\def\Com{\mathsf{Com}}
\newcommand{\pIG}{\mathsf{pIG}}
\newcommand{\IG}{\mathsf{IG}}
\newcommand{\Graphs}{\mathsf{Graphs}}

\def\MC{\mathop{\mathsf{MC}}\nolimits}
\def\top{\mathcal T\!op}
\def\Mod{\mathcal Mod}
\def\id{\mathds{1}}

\def\leq{\leqslant}
\def\geq{\geqslant}

\newtheorem{theorem}{Theorem}[section]
\newtheorem*{theorem*}{Theorem}
\newtheorem{lemma}[theorem]{Lemma}

\newtheorem{proposition}[theorem]{Proposition}
\newtheorem{corollary}[theorem]{Corollary}
\theoremstyle{definition}
\newtheorem{construction}[theorem]{Construction}
\newtheorem{definition}[theorem]{Definition}

\newtheorem{remark}[theorem]{Remark}
\numberwithin{equation}{section}

\newtheorem*{futuretheoreminner}{Theorem~\thefuturetheoreminner}
\newcommand{\thefuturetheoreminner}{} 

\ExplSyntaxOn

\prop_new:N \g_alevel_future_prop

\NewDocumentEnvironment{futuretheorem}{ m o +b }
{
    \renewcommand{\thefuturetheoreminner}{\ref{#1}}
    \IfNoValueTF{#2}
    {\futuretheoreminner}
    {\futuretheoreminner[#2]}
    {\emph{#3}}
}
{
    \endfuturetheoreminner
    \prop_gput:Nnn \g_alevel_future_prop { #1 } { #3 }
    \IfValueT{#2}{ \prop_gput:Nnn \g_alevel_future_prop { #1-attr } { #2 } }
}

\NewDocumentCommand{\pasttheorem}{m}
{
    \prop_if_in:NnTF \g_alevel_future_prop { #1-attr }
    {
        \begin{theorem}[\prop_item:Nn \g_alevel_future_prop { #1-attr }]
    }
    {
        \begin{theorem}
    }
    \label{#1}
    \prop_item:Nn \g_alevel_future_prop { #1 }
    \end{theorem}
}

\ExplSyntaxOff

\renewenvironment{proof}[1][{\itshape Proof}]{{\itshape #1. }}{\qed}

\title{Embedding calculus for parallelized manifolds}

\author{Semyon Abramyan}
\address{Department of Mathematics, ETH Zürich, Rämistrasse 101, 8092 Zürich, Switzerland}
\email{semyon.abramyan@gmail.com}

\begin{document}

\begin{abstract}
We study a variant of the embedding functor $\emb(M, N)$ that incorporates homotopical data from the frame bundle of the target manifold $N$. Given a parallelized $m$-manifold $M$ and an $n$-manifold $N$ equipped with a section of its $m$-frame bundle, we define a modified embedding functor $\widetilde\emb(M, N)$ that interpolates between the standard embedding and a reference framing. Using the manifold calculus of functors, we identify the Taylor tower of $\widetilde\emb(M, N)$ with a mapping space of right modules over the Fulton–MacPherson operad. We prove a convergence theorem under a codimension condition, establishing a weak equivalence between $\widetilde\emb(M, N)$ and its Taylor approximation. Finally, under rationalization, we describe the derived mapping space in terms of a combinatorial hairy graph complex, enabling computational access to the rational homotopy type of the space of embeddings.
\end{abstract}

\maketitle
\tableofcontents

\section{Introduction}

Let $M$ be a parallelised manifold of dimension~$m$, and let $N$ be a smooth manifold of dimension~$n$ that admits a section of the bundle $\Frm(N)$ of $m$-frames on $N$. Denote by $\mathcal O(M)$ the poset of open subsets of $M$ ordered by inclusion.

Introduced by Weiss in \cite{weis99} the \emph{manifold calculus of functors} gives a way to study the homotopy type of functors $F\colon \mathcal O(M) \to \top$ which take isotopy equivalences to weak equivalences. For such a functor Goodwillie, Klein and Weiss define a \emph{Taylor tower}
\begin{center}
\begin{tikzcd}
    & F
    \arrow[dl]
    \arrow[d]
    \arrow[dr]
    \arrow[drr]
    \\
    T_0 F
    &
    T_1 F
    \arrow[l]
    &
    T_2 F
    \arrow[l]
    &
    \arrow[l]
    T_3 F
    &
    \cdots
    \arrow[l]
\end{tikzcd}
\end{center}
of \emph{polynomial approximations} of $F$. Nowadays it is clear that there is a deep relation between the manifold calculus of functors and the operad of little discs~$L\mathbb D_m$. Namely, convergence results of Goodwillie-Weiss~\cite{go-we99} (see also~\cite[Theorem~2.1]{turc13}) imply that if $\dim N - \dim M \geqslant 3$ there are weak equivalences
\[
    \emb(M, N) \xrightarrow{\sim}
    T_\infty\emb(M, N) \xrightarrow{\sim} \map^h_{\mathrm{mod}\text-\mathcal F_m}(\mathcal F_M, \mathcal F^{m\text{-fr}}_N),
\]
where $\mathcal F_m$, $\mathcal F_M$ and $\mathcal F^{m\text{-fr}}_N$ refer to the Fulton-MacPherson operad of $\R^m$, the Axelrod-Singer-Fulton-MacPherson completion of the configuration space of points on $M$ and it's $m$-framed version on~$N$, respectively.
A different (though very similar) incarnation of the embedding functor was studied in \cite{ar-tu14, f-t-w20}. Namely, it was shown that for the embeddings modulo immersions functor
\[
    \overline\emb(M, \R^n) \coloneqq \hofib \big(\emb(M, \R^n) \to \imm(M, \R^n)\big)
\]
there are weak equivalences
\[
    \overline\emb(M, \R^n) \xrightarrow{\sim}
    T_\infty\overline\emb(M, \R^n) \xrightarrow{\sim}
    \map^h_{\mathrm{mod}\text-\mathcal F_m}(\mathcal F_M, \mathcal F_n).
\]
The rational homotopy type of the latter space can be described purely combinatorially (see below).

We consider a slight modification of $\overline\emb$ that allows us to consider a bit more general target manifold, but at the same time is still controllably close to the original embedding functor.

\begin{definition}[$\widetilde\emb(M, N)$]\label{emb-tilde}
Let $M$ be a parallelized manifold of dimension~$m$, and let $N$ be a smooth manifold of dimension~$n$ with a fixed section $\sigma_{std}$ of the bundle $\Frm(N)$ of $m$-frames on $N$.

An embedding $f\colon M \hookrightarrow N$ gives us two sections of the induced bundle $f^*\Frm(N)$ over $M$. The first section is the $m$-frame defined by $df$. The second one is the composite of $f$ with the section $\sigma_{std}$.

Let $\widetilde\emb(M, N)$ (resp. $\widetilde\imm(M, N)$) be the set of pairs $(f, h)$, where $f\colon M \hookrightarrow N$ is an embedding (resp. an immersion) and $h\colon [0,1] \to \Gamma(M, f^*\Frm(N))$ is a path from $\sigma_{std}$ to $df$.

In the similar fashion we define $\widetilde{\mathcal F}_N$ to be a right $\mathcal F_m$-module with the $r$-arity component $\widetilde{\mathcal F}_N(r)$ defined as follows. The space $\widetilde{\mathcal F}_N(r)$ is the space of pairs $\big((x_1, \ldots, x_r), (h_1, \ldots, h_r)\big)$, where $(x_1, \ldots, x_r) \in \mathcal F_N^{m-fr}$ is an $m$-framed configuration on $N$, and $h_i$, $i = 1, \ldots, r$ is a deformation of the $m$-frame at $x_i$ terminating at $\sigma_{std}$. The right $\mathcal F_m$-module structure is given by acting naturally on the first component and duplicating the deformation. 
\end{definition}

Using results of~\cite{ar-tu14, turc13} we prove in this paper the following theorem.\!\footnote{%
Here and further on, $\mathrm{mod}_{\leqslant k}\text-\mathcal F_m$ denotes the category of $k$-truncated right $\mathcal F_m$-modules.%
}

\begin{futuretheorem}{main-limit}
	In the above notation there is a natural equivalence for all $k\leqslant \infty$
	\[
	T_k\widetilde\emb_N(U) \simeq \map^h_{\mathrm{mod}_{\leqslant k}\text-\mathcal F_m} \big(\mathcal F_M, \widetilde{\mathcal F}_N\big).
	\]
\end{futuretheorem}

Moreover we have a Goodwille-Weiss spirit convergence result for $\widetilde\emb$.

\begin{futuretheorem}{main-convergence}
	Let $M$ be a parallelized manifold of dimension~$m$.
	Let $N$ be a smooth manifold of dimension~$n$ with a fixed $m$-frame
	$\sigma_{std}\colon N \to \Frm(N)$. Assume that $n - m \geqslant 3$.
	Then the limit map
	\[
	\widetilde\emb(M,N) \xrightarrow{\sim} T_\infty\widetilde\emb(M, N)
	\]
	is a weak equivalence.
\end{futuretheorem}

Finally, to connect to the computational side we consider the canonical morphism 
\[
\map^h_{\mathrm{mod}\text-\mathcal F_m}(\mathcal F_M, \mathcal F_N)
\to
\map^h_{\mathrm{mod}\text-\mathcal F_m}(\mathcal F_M, \mathcal F_N^{\mathbb Q})
\]
induced by the rationalization $\mathcal F_N \to \mathcal F_N^\Q$. Under certain assumptions on the manifolds, the morphism between the derived mapping spaces above is expected to be a component-wise rational homotopy equivalence (see~\cite[Theorem~1.2]{f-t-w20} for the analogous result). The target can be described combinatorially using the following theorem.

\begin{futuretheorem}{main-computation}
Let $M$ and $N$ be parallelized manifolds of dimensions $m$ and $n$, respectively. Assume that $n - m \geqslant 2$.
Then there is a weak equivalence
\[
	\map^h_{\mathrm{mod}\text-\mathcal F_m} (\mathcal F_M, \mathcal F_N^{\mathbb Q})
	\simeq
	\mc(\hgc^Z_{A_M, H^{\sbullet}(N), n}).
\]
See \textup{\hr[Section]{hairy-graphs-section}} for the definition of the hairy graph complex~$\hgc_{U, V, n}$.
\end{futuretheorem}

\subsection*{Acknowledgments} I am grateful to my supervisor, Thomas Willwacher, for proposing the problem and for the countless hours of helpful discussions. I would also like to thank Victor Turchin and Greg Arone for their guidance on embedding calculus. 

The work was furthermore supported by the NCCR Swissmap, funded by the
Swiss National Science Foundation.

\section{Preliminaries}

\subsection{Topological $W$-construction}
The \emph{$W$-construction} is a functorial cofibrant resolution for operads. It was introduced in~\cite{boar71}. Here we modify the $W$-construction for modules over a topological operad. But first let us recall the original operadic version.

\begin{construction}[Classical Boardman-Vogt resolution]\label{boardman-resolution}
Let $\mathcal P$ be a topological operad. Let $Tree_k$ be the set of isomorphism classes of plain trees with $k$ leaves. For each tree~$\tau$ denote by $V(\tau)$ the set of its vertices and by $E(\tau)$ its set internal edges. For a vertex $v\in V(\tau)$ let $star(v)$ be the set of edges incoming to~$v$.

Let $\mathcal T_k(\mathcal P)$ be the space of ordered trees with $k$ leaves, vertices labeled by~$\mathcal P$, and internal edges labeled by an element of $I = [0,1]$:
\[
    \mathcal T_k(\mathcal P) \coloneqq \bigsqcup_{\tau \in Tree_k} \Big(\prod_{v\in V(\tau)}\mathcal P\big(star(v)\big)\times \prod_{e\in E(\tau)} [0,1] \Big).
\]
The space $W\mathcal P(k)$ is the quotient of $\mathcal T_k(\mathcal P)$ by the followings relations:
\begin{itemize}[leftmargin=0.03\textwidth]
\item Suppose that $\tau\in \mathcal T_k(\mathcal P)$, $v\in V(\tau)$ is a vertex of valence~$n$ labeled by $p\in \mathcal P(n)$, the subtrees stemming from~$v$ are $\tau_1 < \cdots < \tau_n$, and $\sigma \in S_n$. Then $\tau$ is equivalent to the element obtained from $\tau$ by replacing~$p$ by $\sigma^{-1}p$ and by permuting the order of the subtrees to $\tau_{\sigma_1} < \cdots < \tau_{\sigma_n}$.

\item If $\tau$ has an edge $e$ of length~$0$, then $\tau$ is equivalent to the tree obtained by contracting~$e$ and (partially) composing the labels of its vertices.

\item If $\tau$ has a vertex~$v$ of valence~$1$ labeled by the unit $\iota\in\mathcal P(1)$ of the operad~$\mathcal P$, then $\tau$ is equivalent to the tree obtained by removing~$v$.
If~$v$ is between two internal edges of lengths $s$ and $t$, then the length of the merged edge is~$s+t-st$.
\end{itemize}

The action of $S_k$ on $W\mathcal P(k)$ is given by permuting the labeling of the leaves of elements in $\mathcal T_k(\mathcal P)$. Finally, the operad structure on $W\mathcal P$ is defined by grafting trees, and by assigning length~$1$ to the new internal edges. A natural ordering of the leaves of the composite is induced. The trivial tree consisting of an edge with no vertices is the identity of~$W\mathcal P$.

\begin{proposition}[{\cite{boar71}}]
Let $\mathcal P$ be a topological operad such that $\{\iota\} \hookrightarrow \mathcal P(1)$ is a cofibration and each $\mathcal P(n)$ is a cofibrant $S_n$-space. Then $W\mathcal P$ is a cofibrant resolution of $\mathcal P$ with the map $W\mathcal P \xrightarrow{\sim} \mathcal P$ contracting the edges and multicomposing the vertex labels.
\end{proposition}
\end{construction}

\begin{construction}[$W$-construction for modules]
In the above notation (see~\hr[Construcion]{boardman-resolution}) let $\mathcal M$ be a module over~$\mathcal P$. For each tree~$\tau$ denote by $\star \in V(\tau)$ its root. 

Define $\mathcal T^{\mathcal P}_k(\mathcal M)$ in the same manner:
\[
    \mathcal T^{\mathcal P}_k(\mathcal M)
    \coloneqq
    \bigsqcup_{\tau \in Tree_k} \Big(\mathcal M\big(star(\star)\big)\times \hspace{-1em}\prod_{v \in V(\tau)\setminus\star} \hspace{-1em} \mathcal P\big(star(v)\big)\times \hspace{-0.5em}\prod_{e\in E(\tau)} [0,1]\Big).
\]
The space $W^{\mathcal P}\!\mathcal M(k)$ is the quotient of $\mathcal T^{\mathcal P}_k(\mathcal M)$ by the same relation with a modification of the second one:
\begin{itemize}[leftmargin=0.03\textwidth]
\item If $\tau$ has an edge $e$ of length~$0$ incoming to the root, then $\tau$ is equivalent to the tree obtained by contracting~$e$ and (partially) acting on the module label with the corresponding operadic label of the second vertex.
\end{itemize}
We will omit the superscript whenever it is clear from the context.

Note that $W\mathcal M(k)$ can be viewed as two-levelled trees with the root decorated by $\mathcal M$ and the second level decorated by $W\mathcal P$. Naturally $W\mathcal M$ has a structure of right $W\mathcal P$-module defined again by grafting trees, and by assigning length~$1$ to the new internal edges.

The map $W\mathcal M \to \mathcal M$ sending a tree $\tau$ to the multicomposition of its vertex labels is a morphism of right $W\mathcal P$-modules. Moreover, it is an arity-wise homotopy equivalence given by contracting edges. 
\end{construction}

{

}

\subsection{Versions of the Fulton-MacPherson operad}
For a manifold $M^m$ let $\conf_k(M)$, $k\geqslant 0$, denote the configuration space
\[
    \conf_k(M) \coloneqq \{(x_1, \ldots, x_k) \in M^{\times k} \,|\; \text{$x_i \neq x_j$ for $i \neq j$}\}.
\]
Let $\mathcal F_M(k)$ be its Axelrod-Singer-Fulton-MacPherson completion (see~\cite{sinh04} for a thorough description). It is a manifold with corners whose interior is $\conf_k(M)$. The boundary strata consist of configuration where some of the points collided. When $M = \R^m$ we obtain the \emph{Fulton-MacPherson operad} $\mathcal F_m$ (see~\cite{ge-jo94, salv01}). If a manifold $N$ has dimension greater or equal to~$m$ we define an $m$-framed version of $\mathcal F_N(k)$ to be a space $\mathcal F^{m\text{-fr}}_N(k)$ which fibres over $\mathcal F_N$ with a fibre over a point $x \in \mathcal F_N(k)$ being the space of tuples $(\alpha_1, \ldots, \alpha_k)$, where $\alpha_i\colon \R^m \to T_{p_i(x)}N$ is a linear injective map. Here $p_i\colon F_N(k) \to N$, $0 \leqslant i \leqslant k$, is the projection to the $i$-th point.
The sequences $\mathcal F^{\text{fr}}_M \coloneqq \mathcal F^{m\text{-fr}}_M$, $\mathcal F^{m\text{-fr}}_N$ are right $\mathcal F^{\text{fr}}_m$-modules. The arity zero component acts by forgetting points in configurations. An element $x \in \mathcal F_m^{\text{fr}}(k)$ acts by replacing a point in a configuration by the infinitesimal configuration~$x$ according to the framing. For a parallelised manifold $M$ the sequence $\mathcal F_M$ is naturally a right $\mathcal F_m$-module. The same holds if a manifold $N$ admits an $m$-frame.

It was shown by Salvatore that $\mathcal F_m$ is weakly equivalent to the little discs operad~$L\mathbb D_m$.

\begin{proposition}[{\cite[Proposition 3.9]{salv01}}]\label{salv-w-eq}
There is a zigzag of homotopy equivalences
\begin{center}
\begin{tikzcd}
    L\mathbb D_m
    &
    W(L\mathbb D_m)
    \arrow[l, swap, "\sim"]
    \arrow[r, "\sim"]
    &
    \mathcal{F}_m.
\end{tikzcd}
\end{center}
\end{proposition}

In this paper we also consider the following version of $\mathcal F_M$.
Suppose that $N$ admits an $m$-frame $\sigma_{std}\colon N \to \Frm(N)$.
We define a \emph{path $m$-framed} version of $\mathcal F_N(k)$ to be a space $\widetilde{\mathcal F}_N(k)$ which fibres over $\mathcal F^{m\text{-fr}}_N(k)$ with a fibre over a point $x\in\mathcal F^{m\text{-fr}}_N(k)$ being the space of tuples $(h_1, \ldots, h_k)$, where $h_i\colon [0,1] \to \Gamma\big(p_i(x), \Frm(N)\big)$ is a path in $m$-frames of~$N$ over $p_i(x)$ starting at the given frame~$\sigma_{std}$.  Then $\widetilde{\mathcal F}_N$ is naturally a right $\mathcal F_m$-module.

\subsection{Model structure on modules over an operad}\label{model-modules}

Let $\mathcal P$ be a topological operad, and suppose that $\mathcal P$ is \emph{$\top$-cofibrant} meaning that $\big\{\mathcal P(r)\big\}_{r \in \Z_{\geqslant 0}}$ consists of cofibrant spaces.
The category $\Mod_{\mathcal P}$ of right $\mathcal P$-modules admits a model structure so that a morphism $f\colon \mathcal M \to \mathcal N$ is a weak equivalence (resp. fibration) if the morphisms $f\colon \mathcal M(r) \to \mathcal N(r)$ are weak equivalences (resp. fibrations) in $\top$.

Let $f\colon \mathcal P \to \mathcal Q$ be a morphism of $\top$-cofibrant operads. If we equip the categories of modules $\Mod_{\mathcal P}$ and $\Mod_{\mathcal Q}$ with the above model structure, we have the following Quillen adjunction.

\begin{theorem}[{see~\cite[Theorem~16.B]{fres09}}]\label{model-ind-res}
The induction and restriction functors
\begin{equation}\label{q-ad}
    \ind^{\mathcal Q}_{\mathcal P}\colon \Mod_{\mathcal P} \rightleftarrows \Mod_{\mathcal Q}\rcolon \res^{\mathcal Q}_{\mathcal P}
\end{equation}
define a Quillen adjunction. Moreover, if $f\colon \mathcal P \to \mathcal Q$ is a weak equivalence, then \eqref{q-ad} is a Quillen equivalence.
\end{theorem}

\subsection{Algebraic $W$-construction}
Here we briefly recall the algebraic version of $W$-construction (see~\cite[Section~5]{f-t-w17}) and introduce its module version.

\begin{construction}[$W$-construction for dg Hopf cooperads~{\cite[Construction~5.1]{f-t-w17}}]\label{w-cooperads}
Let $\mathcal C$ be a dg Hopf cooperad with $\mathcal C(0) = 0$ and $\mathcal C(1) = \Q$. Denote by $\bar{\mathcal C}$ its coaugmentation coideal which is given by $\overline{\mathcal C}(0) = \overline{\mathcal C}(1) = 0$ and $\overline{\mathcal C}(r) = \mathcal C (r)$ for $r \geqslant 2$. For finiteness conditions we consider now the set $Tree'_k \subseteq Tree_k$ formed by trees whose vertices have at least two incoming edges.
As before, we define the $W$-construction in two steps, essentially by just dualising objects.

We start from the algebra $\mathcal T_k(\mathcal C)$ of decorations of ordered trees with~$k$ leaves where the vertices are decorated by the cooperad~$\mathcal C$ and the edges are decorated by polynomial forms~$\Q[t, \mathrm dt]$ on the unit interval:
\[
	\mathcal T_k(\mathcal C) \coloneqq \prod_{\tau \in Tree'_k} \Big(\bigotimes_{v \in V(\tau)} \mathcal C\big(star(v)\big)\otimes \bigotimes_{e\in E(\tau)} \Q[t, \mathrm dt]\Big).
\]
The algebra $W\mathcal C(k)$ is the subalgebra of decorations of $T_k(\mathcal C0)$ satisfying the following properties:
\begin{itemize}[leftmargin=0.03\textwidth]
	\item (Equivariance condition) The obvious modification of the first relation in~\hr[Construction]{boardman-resolution}.
	
	\item (Contraction condition) Let $e \in E(\tau)$ be an internal edge of~$\tau$. Denote by $v$ the vertex of $\tau/e$ obtained by contracting the edge~$e$. Then the values of decoration $\xi$ on $\tau$ and $\tau/e$ are related by the formula
	\[
		\Delta_e \xi_{\tau/e} = \mathrm{ev}^e_{t = 0}\xi_{\tau},
	\]
	where $\Delta_e$ denotes the cocomposition applied to the vertex $v$ in $\xi_{\tau/e}$ and $\mathrm{ev}^e_{t = 0}$ is the evaluation at $t=0$ applied to the edge $e$ in $\xi_\tau$.
\end{itemize}
The differential on $W\mathcal C$ is induced by the differentials on $\mathcal C$ and $\Q[t, \mathrm dt]$. The commutative algebra structure is given by the pointwise multiplication of the decorations $\xi\colon \tau \to \xi_{\tau}$ in the commutative dg algebras $\bigotimes_{v \in V(\tau)} \mathcal C\big(star(v)\big)\otimes \bigotimes_{e\in E(\tau)} \Q[t, \mathrm dt]$.

The cocompostion on $W\mathcal C$ is defined by a set of maps
\[
	\Delta_*\colon W\mathcal C(k) \to W\mathcal C(k'+1)\otimes W\mathcal C(k''),
\]
for each decomposition $k = k' + k''$.
Note that the target is spanned by decorations defined on pairs of trees $(\tau', \tau'') \in Tree'_{k'+1}\times Tree'_{k''}$ which satisfy above conditions with respect to both variables $\tau'$ and $\tau''$. Finally, for $\xi \in W\mathcal C(k)$ we set
\[
	\Delta_*\xi (\tau', \tau'') \coloneqq \mathrm{ev}^{e_*}_{t = 1} \xi(\tau'\circ_*\tau'').
\]
Here $\tau'\circ_*\tau''$ is the tree obtained by grafting the root of $\tau''$ to the leaf of $\tau$ indexed by~$*$, and $\mathrm{ev}^{e_*}_{t = 1}$ is the evaluation of internal edge~$e_*$ produced by grafting.
\end{construction}

There is a canonical morphism of dg Hopf cooperads $\rho\colon \mathcal C \to W\mathcal C$. It takes an element $c \in \mathcal C(k)$ to the decoration such that
$\rho(c)(\tau) = \overline\Delta_\tau(c)\otimes 1^{\otimes E(\tau)}$.
Here $\overline\Delta_\tau(c)$ is the reduced tree-wise coproduct of~$c$, and we take the constant edge decoration being equal to~$1$.

\begin{proposition}[{\cite[Section~5]{f-t-w17}}]
Let $\mathcal C$ be a reduced dg Hopf $\Lambda$-cooperad. Then there is a natural $\Lambda$-structure on $W\mathcal C$ such that the morphism
\[
	\rho\colon \mathcal C \to W\mathcal C
\]
defines a fibrant resolution of dg Hopf $\Lambda$-cooperads.
\end{proposition}

\begin{construction}[$W$-construction for comodules]
In the above notation (see~\hr[Construction]{w-cooperads}) let $\mathcal M$ be a $\mathcal C$-comodule. As before, we define
\[
	\mathcal T^{\mathcal C}_k(\mathcal M) \coloneqq \prod_{\tau \in Tree'_k}
	\Big(\mathcal M\big(star(\star)\big)\otimes \bigotimes_{v \in V(\tau)\setminus \star} \mathcal C\big(star(v)\big) \otimes \bigotimes_{e\in E(\tau)} \Q[t, \mathrm dt]\Big).
\]
The $W$-construction $W^{\mathcal C}\mathcal M(k)$ is the subalgebra of $T^{\mathcal C}_k(\mathcal M)$ satisfying the same relations as above with an modification of the second:
\begin{itemize}[leftmargin=0.03\textwidth]
	\item (Contraction condition) If $e \in E(\tau)$ is an internal edge of~$\tau$ incoming to the root, then the values of decoration $\xi$ on $\tau$ and $\tau/e$ are related by the formula
	\[
	\Delta_{M, e} \xi_{\tau/e} = \mathrm{ev}^e_{t = 0}\xi_{\tau},
	\]
	where $\Delta_{M, e}$ denotes the $\mathcal C$-coaction applied to the root~$\star$ in $\xi_{\tau/e}$.
\end{itemize}
The differential on $W\mathcal M$ is induced by the differentials on $\mathcal M$, $\mathcal C$ and $\Q[t, \mathrm dt]$. The commutative dg algebra structure is again given by the pointwise multiplication.

The $W\mathcal C$-coaction is defined by the same formula as above. 
\end{construction}

An obvious modification of the proof of~\cite[Proposition~5.2]{f-t-w17} leades to the following proposition.

\begin{proposition}
	The canonical morphism $\rho\colon \mathcal M \to W\mathcal M$ is a weak-equivalence of $W\mathcal C$-comodules.
\end{proposition}

\subsection{Graph complexes and graph operads} 
We briefly recall the definition of Kontsevich graph cooperad $\Graphs_n$ (see~\cite{kont99}).
An \emph{admissible graph} with~$r$ external and~$k$ internal vertices is an undirected graph such that
\begin{itemize}[leftmargin=0.03\textwidth]
	\item the external vertices are numbered by~$1, \ldots, r$;
	
	\item there is at least one external vertex in every connected component;
	
	\item every internal vertex has valence at least~$3$.
\end{itemize}
Tadpoles and multiple edges are allowed. Here is an example of an admissible graph.
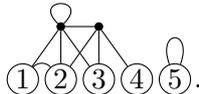
\[
\begin{tikzpicture}
	\node[ext] (v1) at (0,0) {1};
	\node[ext] (v2) at (0.5,0) {2};
	\node[ext] (v3) at (1,0) {3};
	\node[ext] (v4) at (1.5,0) {4};
	\node[ext] (v5) at (2,0) {5};
	\node[int] (w1) at (0.5,.7) {};
	\node[int] (w2) at (1.0,.7) {};
	\draw (v1) to [out=50,in=130] (v2) (v1) edge (w1)  (v2) edge (w1) edge (w2) (v3) edge (w1) edge (w2) (v4) edge (w2) (w1) edge (w2) (v5) to [out=110,in=70, loop] (v5) (w1) to [loop] (w1);
\end{tikzpicture}.
\]
The cohomological degree of a graph is
\[
	(n-1)(\#\text{edges})-n (\#\text{internal vertices}).
\]
An $n$-orientation on an admissible graph is the following:
\begin{itemize}[leftmargin=0.03\textwidth]
	\item For even~$n$ it is an ordering of the set of edges up to even permutations.
	
	\item For odd~$n$ it is an ordering of the set of half-edges and internal vertices up to even permutations.
\end{itemize}
An admissible graph with orientation data is called \emph{oriented graph}. Note that we mostly omit the orientation data in pictures, leaving the sign undefined.

The space $\Graphs_n(r)$ is defined to be the space of $\mathbb Q$-linear combinations of isomorphism classes of \mbox{($n$-)oriented} admissible graphs with~$r$ external vertices modulo the identification of an oriented graph with minus the same graph with the opposite orientation.

Each space $\Graphs_n(r)$ is a diffential graded commutative algebra. The product is obtained by gluing graphs along the external vertices:
\begin{equation}\label{equ:Graphs product pic}
	\left(
	\begin{tikzpicture}
		\node[ext] (v1) at (0,0) {$\scriptstyle 1$};
		\node[ext] (v2) at (.7,0) {$\scriptstyle 2$};
		\node[ext] (v3) at (1.4,0) {$\scriptstyle 3$};
		\node[int] (i1) at (.35,.7) {};
		\node[int] (i2) at (1.05,.7) {};
		\draw (v1) edge (i1) 
		(i1) edge (i2) edge (v2) 
		(i2) edge (v2) edge (v3);
	\end{tikzpicture}
	\right)
	\wedge 
	\left(
	\begin{tikzpicture}
		\node[ext] (v1) at (0,0) {$\scriptstyle 1$};
		\node[ext] (v2) at (.7,0) {$\scriptstyle 2$};
		\node[ext] (v3) at (1.4,0) {$\scriptstyle 3$};
		\node[int, white] (i1) at (.35,.7) {};
		\draw 
		(v2) edge[bend left] (v3);
	\end{tikzpicture}
	\right)
	=
	\begin{tikzpicture}
		\node[ext] (v1) at (0,0) {$\scriptstyle 1$};
		\node[ext] (v2) at (.7,0) {$\scriptstyle 2$};
		\node[ext] (v3) at (1.4,0) {$\scriptstyle 3$};
		\node[int] (i1) at (.35,.7) {};
		\node[int] (i2) at (1.05,.7) {};
		\draw (v1) edge (i1) 
		(i1) edge (i2) edge (v2) 
		(i2) edge (v2) edge (v3)
		(v2) edge[bend left] (v3);
	\end{tikzpicture}\, .
\end{equation}
To fix the signs in such pictures one has to specify the orientation data on the right-hand side. We do this by juxtaposing the natural order of edges or vertices on the left-hand side.

The diffirential is given by contracting an edge between two distinct vertices at least one of which is internal:
\begin{align*}
	\mathrm d
	\begin{tikzpicture}[baseline=-.65ex]
		\node[ext] (v) at (0,0) {$i$};
		\node[int](w) at (0,.5) {};
		\draw (v) edge +(-.5,.5) edge +(.5,.5) edge (w) (w) edge +(-.2,.5) edge +(.2,.5);
	\end{tikzpicture}
	&=
	\begin{tikzpicture}[baseline=-.65ex]
		\node[ext] (v) {$i$};
		\draw (v) edge +(-.5,.5) edge +(-.2,.5) edge +(.2,.5) edge +(.5,.5);
	\end{tikzpicture}
	&
	\mathrm d
	\begin{tikzpicture}[baseline=-.65ex]
		\node[int] (v) at (0,0) {};
		\node[int](w) at (0,.3) {};
		\draw (v) edge +(-.5,.5) edge +(.5,.5) edge (w) (w) edge +(-.2,.5) edge +(.2,.5);
	\end{tikzpicture}
	&=
	\begin{tikzpicture}[baseline=-.65ex]
		\node[int] (v) {};
		\draw (v) edge +(-.5,.5) edge +(-.2,.5) edge +(.2,.5) edge +(.5,.5);
	\end{tikzpicture}.
\end{align*}
Note that each dg commutative algebra $\Graphs_n(r)$  is quasi-free, generated by the internally connected graphs $\IG_n(r) \subseteq \Graphs_n(r)$, i.e. graphs that remain connected after we remove the external vertices.

Furthermore, the collection of spaces $\Graphs_n(r)$ assembles into a dg Hopf $\Lambda$-cooperad. To define the cooperadic cocomposition, it is sufficient to specify the reduced cocompositions
\[
	\Delta_s\colon \Graphs_n(r) \to \Graphs_n(r-s+1)\otimes \Graphs_n(s)
\]
corresponding to the subset $\{1, \ldots, s\} \subseteq \{1, \ldots, r\}$.
For a graph $\Gamma \in \Graphs_n(r)$
\[
	\Delta_s(\Gamma) \coloneqq \sum_{\substack{\gamma \subseteq \Gamma \\1, \ldots, s \in \gamma}} \pm (\Gamma/\gamma)\otimes \gamma,
\]
with the sum over all subgraphs $\gamma \subseteq \Gamma$ that contain the external vertices $1, \ldots, s$ and no other external vertices, and with $\Gamma/\gamma$  the graph with $\gamma$ contracted to a new external vertex numbered~$1$ and the natural ordering of the remaining vertices. The sign is the sign of the unshuffle permutation moving the edges/vertices of $\gamma$ to the right relative to the order of edges/vertices in~$\Gamma$. The $\Lambda$-operations $\Graphs_n(r) \to \Graphs_n(r+1)$ are defined by adding a zero-valent external vertex to the graph.
Finally, the right $S_r$-action is defined by permutations of the external vertices.

\begin{theorem}[Kontsevich, Lambrechts-Voli\'c]
For every $n \geq 2$ there is a natural map
\[
	\Graphs_n \to H^{\sbullet}(\mathcal F_n),
\]
which is a quasi-isomorphism. 
\end{theorem}

\subsection{The $\Graphs_n$-comodule $\Graphs_{V, n}$}
Let $V$ be a finite dimensional positively graded vector space and $n$ an integer. Define $\Graphs_{V, n}(r)$ to be the space of $\mathbb Q$-linear combinations of isomorphism classes of oriented admissible graphs with $r$ external vertices, where all the vertices are decorated by $S(V)$, with each decoration in $V$ counting $+1$ to the valency.
\[
\begin{tikzpicture}
	\node[ext, label=90:{$\scriptstyle \beta\gamma$}] (v1) at (0,0) {$\scriptstyle 1$};
	\node[ext] (v2) at (.7,0) {$\scriptstyle 2$};
	\node[ext] (v3) at (1.4,0) {$\scriptstyle 3$};
	\node[int] (i1) at (.35,.7) {};
	\node[int, label=90:{$\scriptstyle \alpha$}] (i2) at (1.05,.7) {};
	\draw (v1) edge (i1) 
	(i1) edge (i2) edge (v2) 
	(i2) edge (v3)
	(v2) edge[bend left] (v3)
	(i1) to [in=75, out=105, loop] (i1);
\end{tikzpicture}
\in \Graphs_{V,n}(3), \text{ with $\alpha,\beta,\gamma\in V$.} 
\]

As before, the graded commutative algebra structure on $\Graphs_{V,n}(r)$ is given by gluing graphs along the external vertices multiplying the corresponding external vertex decorations. The differential, the $S_r$-action and the $\Lambda$-structure are defined as before.

Finally, there is a $\Graphs_n$-comodule structure on $\Graphs_{V,n}$ defined by subgraph contraction, for example;
\[
\begin{tikzpicture}
	\node[ext, label=90:{$\scriptstyle \alpha$}] (v1) at (0,0) {$\scriptstyle 1$};
	\node[ext, label=90:{$\scriptstyle \beta$}] (v2) at (.7,0) {$\scriptstyle 2$};
	\node[ext] (v3) at (1.4,0) {$\scriptstyle 3$};
	\node[int, white] (i1) at (.35,.7) {};
	\draw (v1) edge (v2)
	(v2) edge (v3);
\end{tikzpicture}
\mapsto
\begin{tikzpicture}
	\node[ext, label=90:{$\scriptstyle \alpha\beta$}] (v1) at (0,0) {$\scriptstyle 1$};
	\node[ext] (v3) at (.7,0) {$\scriptstyle 2$};
	\draw (v1) edge (v3)
	(v1) to [in=-75, out=-105, loop] (v1);
\end{tikzpicture}
\otimes 
\begin{tikzpicture}
	\node[ext] (v1) at (0,0) {$\scriptstyle 1$};
	\node[ext] (v3) at (.7,0) {$\scriptstyle 2$};
\end{tikzpicture}
+
\begin{tikzpicture}
	\node[ext, label=90:{$\scriptstyle \alpha\beta$}] (v1) at (0,0) {$\scriptstyle 1$};
	\node[ext] (v3) at (.7,0) {$\scriptstyle 2$};
	\draw (v1) edge (v3);
\end{tikzpicture}
\otimes 
\begin{tikzpicture}
	\node[ext] (v1) at (0,0) {$\scriptstyle 1$};
	\node[ext] (v3) at (.7,0) {$\scriptstyle 2$};
	\draw (v1) edge (v3);
\end{tikzpicture}
\]

\subsection{Hairy graph complexes}\label{hairy-graphs-section}
Let $U$, $V$ be a pair of finite dimensional positively graded vector spaces and $n$ an integer. Define $\hgc_{U, V, n}$ to be the space of $\Q$-linear combinations of isomorphism classes of admissible graphs with external vertices of valence~$1$, where all the vertices are decorated by $S(V)$, with each decoration in $V$ counting $+1$ to the valence, and the external vertices are decorated by $U_1^* = (\Q 1 \oplus U)^*$, where $1$ is a formal element of degree~$0$. For the comprehensive exposition we refer to \cite[Section~9.2]{will23}

\[
\begin{tikzpicture}[scale=.7,baseline=-.65ex]
	\node[int] (v1) at (-1,0){};
	\node[int,label={$v$}] (v2) at (0,1){};
	\node[int] (v3) at (1,0){};
	\node[int] (v4) at (0,-1){};
	\node[ext,label=180:{$a_1$}] (w1) at (-2,0) {};
	\node[ext,label=0:{$a_2$}] (w2) at (2,0) {};
	\node[ext,label=270:{$a_3$}] (w3) at (0,-2) {};
	\draw (v1)  edge (v2) edge (v4) edge (w1) (v2) edge (v4) (v3) edge (v2) edge (v4) (v4) edge (w3) (v3) edge (w2);
\end{tikzpicture}
\,,\quad\quad
a_1,a_2,a_3\in U_1^*, v\in V.
\]

\section{Operadic part}

\subsection{Setting up}
Let $M$ be a parallelized manifold of dimension~$m$, and let $N$ be a smooth manifold of dimension~$n$ that admits a section of the bundle $\Frm(N)$ of $m$-frames on $N$. Denote the resulting $m$-frame on $N$ by $\sigma_{std}$.

Any embedding $f\colon M \hookrightarrow N$ gives us two sections of the induced bundle $f^*\Frm(N)$ over $M$. The first section is the $m$-frame defined by $df$. The second is the composite of $f$ with the section $\sigma_{std}$.

Let $\widetilde\emb(M, N)$ (resp. $\widetilde\imm(M, N)$) be the set of pairs $(f, h)$, where $f\colon M \hookrightarrow N$ is an embedding (resp. an immersion) and $h\colon [0,1] \to \Gamma(M, f^*\Frm(N))$ is a path from $\sigma_{std}$ to $df$.

\subsection{The limit of the Taylor tower for $\widetilde\emb(M, N)$}
Here we give a description of the Taylor tower for the functor $\widetilde\emb(\phantom M, N)\colon \mathcal O(M) \to \top$.

\pasttheorem{main-limit}

\begin{proof}
Note that $\widetilde\emb(M, N)$ is a pullback of the following diagram
\[
\begin{tikzcd}
\widetilde\emb(M, N)
\arrow[d] \arrow[r]
\arrow[dr, phantom, "\ulcorner", very near start]
&
\widetilde\imm(M, N)
\arrow[d]
\\
\emb(M,N)
\arrow[r]
&
\imm(M,N).
\end{tikzcd}
\]
Since $\widetilde\imm(\phantom M, N)\colon \mathcal O(M) \to \top$ is a linear functor and the canonical map $\widetilde\emb(\phantom M, N) \to \widetilde\imm(\phantom M, N)$ is a weak equivalence once restricted to a disc,
the diagram can be written as
\[
\begin{tikzcd}
\widetilde\emb(M, N)
\arrow[d] \arrow[r]
\arrow[dr, phantom, "\ulcorner", very near start]
&
T_1\widetilde\emb(M, N)
\arrow[d]
\\
\emb(M,N)
\arrow[r]
&
T_1\emb(M, N).
\end{tikzcd}
\]
Therefore, $T_k\widetilde\emb(M,N)$ can be described as a pullback
\[
\begin{tikzcd}
T_k\widetilde\emb(M, N)
\arrow[d] \arrow[r]
\arrow[dr, phantom, "\ulcorner", very near start]
&
T_1\widetilde\emb(M, N)
\arrow[d]
\\
T_k\emb(M,N)
\arrow[r]
&
T_1\emb(M, N).
\end{tikzcd}
\]

To proceed we show that the diagram
\[
T_k\emb(M,N) \to T_1\emb(M,N) \gets T_1\widetilde\emb(M,N)
\]
is weakly equivalent to
\[
\map^h_{\mathrm{mod}_{\leqslant k}\text-\mathcal F_m} \big(\mathcal F_M, \mathcal F^{m\text{-fr}}_N\big)
\to
\map^h_{\mathrm{mod}_{\leqslant 1}\text-\mathcal F_m} \big(\mathcal F_M, \mathcal F^{m\text{-fr}}_N\big)
\gets
\map^h_{\mathrm{mod}_{\leqslant 1}\text-\mathcal F_m} \big(\mathcal F_M, \widetilde{\mathcal F}_N\big).
\]
The weak equivalence between first two terms is a part of the \cite{go-we99, weis99} convergence result. The latter as noted above is weakly equivalent to $\widetilde\imm(M,N)$, which, in turn, is weakly equivalent to
\[
	\map_{\mathrm{mod}_{\leqslant 1}\text-L\mathbb D^{\text{fr}}_m}^{h} \big(\emb(\phantom M, M), \widetilde\emb(\phantom M, N)\big),
\]
where the right $L\mathbb D_m^{\mathrm{fr}}$-module structure on $\widetilde\emb(\phantom M, N)$ is given, as usual, by restriction to disjoint copies of $\mathbb D^m$.
Finally, the usual compactification argument implies a weak equivalence with
\[
\map_{\mathrm{mod}_{\leqslant 1}\text-\mathcal F_m}^{h} \big(\mathcal F_M, \widetilde{\mathcal F}_N\big).
\]

To conclude the proof, we need to show that a pullback
\[
\begin{tikzcd}
\arrow[dr, phantom, "\ulcorner", very near start]
&
\map^h_{\mathrm{mod}_{\leqslant 1}\text-\mathcal F_m} \big(\mathcal F_M, \widetilde{\mathcal F}_N\big)
\arrow[d]
\\
\map^h_{\mathrm{mod}_{\leqslant k}\text-\mathcal F_m} \big(\mathcal F_M, \mathcal F^{m\text{-fr}}_N\big)
\arrow[r]
&
\map^h_{\mathrm{mod}_{\leqslant 1}\text-\mathcal F_m} \big(\mathcal F_M, \mathcal F^{m\text{-fr}}_N\big)
\end{tikzcd}
\]
is given by $\map_{\mathrm{mod}_{\leqslant k}\text-\mathcal F_m}^h \big(\mathcal F_M, \widetilde{\mathcal F}_N\big).$

Let $\mathcal F^h_M \to \mathcal F_M$ be the "hairy" cofibrant replacement (see~\cite[p.~1252]{turc13}). Since with the projective model structure (\hr[Section]{model-modules}) every module is fibrant, the "derived" diagram above can be written with (non-derived) mapping spaces as
\[
\begin{tikzcd}
\arrow[dr, phantom, "\ulcorner", very near start]
&
\map_{\mathrm{mod}_{\leqslant 1}\text-\mathcal F_m} \big(\mathcal F^h_M, \widetilde{\mathcal F}_N\big)
\arrow[d]
\\
\map_{\mathrm{mod}_{\leqslant k}\text-\mathcal F_m} \big(\mathcal F^h_M, \mathcal F^{m\text{-fr}}_N\big)
\arrow[r]
&
\map_{\mathrm{mod}_{\leqslant 1}\text-\mathcal F_m} \big(\mathcal F^h_M, \mathcal F^{m\text{-fr}}_N\big).
\end{tikzcd}
\]
Finally, the pullback above is isomorphic to
\[
\map_{\mathrm{mod}_{\leqslant k}\text-\mathcal F_m} \big(\mathcal F^h_M, \widetilde{\mathcal F}_N\big) \simeq \map_{\mathrm{mod}_{\leqslant k}\text-\mathcal F_m}^h \big(\mathcal F_M, \widetilde{\mathcal F}_N\big).
\]
Indeed, the underlying morphism $\mathcal F^h_M \to \mathcal F^{m\text{-fr}}_N$ of $\mathcal F_m$-modules $\mathcal F^h_M$ and $\widetilde{\mathcal F}_N$ is uniquely defined by the projection onto
$\map_{\mathrm{mod}_{\leqslant k}\text-\mathcal F_m} \big(\mathcal F^h_M,
\mathcal F^{m\text{-fr}}_N\big)$
and the path factor at each point is uniquely defined by the projection onto
$\map_{\mathrm{mod}_{\leqslant 1}\text-\mathcal F_m} \big(\mathcal F^h_M,
\widetilde{\mathcal F}_N\big)$.
\end{proof}

\subsection{Convergence. Proof of \hr[Theorem]{main-convergence}}
In this section, we prove our main convergence result.

\pasttheorem{main-convergence}

\textit{Idea of the proof}.
We construct the diagram
\begin{equation}\label{theorem32-diagram}
\begin{tikzcd}
	F
	\arrow[d]
	&&&
	F'
	\arrow[d]
	\\
	\widetilde\emb(M, N)
	\arrow[d, two heads, "p"]
	\arrow[r]
	&
	T_\infty \widetilde\emb(M, N)
	\arrow[d]
	\arrow[r, "\sim"]
	&
	\map^h_{\mathrm{mod}\text-\mathcal F_m}(\mathcal F_M, \widetilde{\mathcal F}_N)
	\arrow[d, two heads]
	\arrow[r, "\sim"]
	&
	\map_{\mathrm{mod}\text-\mathcal F_m}(\mathcal F^h_M, \widetilde{\mathcal F}_N)
	\arrow[d, two heads]
	\\
	\emb(M, N)
	\arrow[r, "\sim"]
	&
	T_\infty\emb(M, N)
	\arrow[r, "\sim"]
	&
	\map^h_{\mathrm{mod}\text-\mathcal F_m}(\mathcal F_M, \mathcal F^{m\text{-fr}}_N)
	\arrow[r, "\sim"]
	&
	\map_{\mathrm{mod}\text-\mathcal F_m}(\mathcal F^h_M, \mathcal F^{m\text{-fr}}_N),
\end{tikzcd}
\end{equation}
where two-headed arrows are fibrations and the induced map between the fibres $F$ and $F'$ over the same connected component is a weak equivalence. The latter implies that the limit map is a weak equivalence.

\vspace{0.5em}
We start with the left column.
In the following lemma we prove that the left bottom arrow~$p$ is a fibration
and describe the fibre.

\begin{lemma} The canonical map $\widetilde\emb(M,N) \to \emb(M, N)$ induced by the projection onto the first factor is a fibration.
The fibre over a given embedding $f \in \emb(M, N)$ is homotopy equivalent to the space of sections $\Gamma\Big(M, \Path^{fib}_{\sigma_{std}, df}\big(f^*\Frm(N)\big)\Big)$,
where $\Path^{fib}_{\sigma_{std}, df}\big(f^*\Frm(N)\big)$ is the fibre-wise space of paths of $f^*\Frm(N)$ with the paths starting at $\sigma_{std}$ and terminating at $df$. The same holds for $\imm$.
\end{lemma}

\begin{proof}
Let $X$ be a topological space, and suppose we are given a (solid) diagram

\begin{center}
\begin{tikzcd}
    X\times \{0\}
    \arrow[d, hook]
    \arrow[r]
    &
    \widetilde\emb(M, N)
    \arrow[d]
    \\
    X\times I
    \arrow[ur, dashed, "\widetilde f"]
    \arrow[r, "f"]
    &
    \emb(M, N).
\end{tikzcd}
\end{center}
We need to construct a dashed arrow. The first factor of $\widetilde f$ is uniquely defined by $f$. To define the second factor one needs to construct a map from $X\times I$ to
$\Path_{\sigma_{std}, df}\big(\Gamma(M, f^*\Frm(N))\big)$. By the exponential law, it is the same as a map from $X\times I^2 \to \Gamma(M, f^*\Frm(N))$ that is defined on the product $X \times (I\times \partial I \cup \{0\}\times I)$ of $X$ with three edges. Let $r \colon I^2 \to (I\times \partial I \cup \{0\}\times I)$ be a retraction defined by the stereographic projection of the square $I^2$ onto the union of three edges from the point $(2, \frac 12)$.
Then a required extension $X\times I^2 \to \Gamma(M, f^*\Frm(N))$ can be obtained by the composite
\[
    X\times I^2 \xrightarrow{\id\times r} X \times (I\times \partial I \cup \{0\}\times I) \to \Gamma(M, f^*\Frm(N)).
\]

To describe the fibre now we need to find the preimage over a point. A point in the preimage is a path in the space $\Gamma(M, f^*\Frm(N))$ of sections starting at $\sigma_{std}$ and terminating at $df$. Thus, the preimage is $\Path_{\sigma_{std}, df}\big(\Gamma(M, f^*\Frm(N))\big)$. The latter space coincides with $\Gamma\Big(M, \Path^{fib}_{\sigma_{std}, df}\big(f^*\Frm(N)\big)\Big)$ from the assertion.
\end{proof}

\vspace{0.5em}
Note that the remaining vertical arrows in~\eqref{theorem32-diagram} are fibrations since the target map $\widetilde{\mathcal F}_N \to \mathcal F^{m\text{-fr}}_N$ is.
Now we pass to the middle horizontal arrows.
The bottom arrow is a weak equivalence due to~\cite[Theorem~2.1]{turc13}.
By \hr[Theorem]{main-limit}, we already know the equivalence
\[
	T_\infty\widetilde{\emb} (M, N) \simeq
	\map^h_{\mathrm{mod}\text-L\mathbb D_m} \big(\semb(\phantom U, M), \widetilde{\emb}(\phantom U, N)\big).
\]
Therefore, we only need to show the equivalence
\[
	\map^h_{\mathrm{mod}\text-L\mathbb D_m} \big(\emb(\phantom U, M), \widetilde{\emb}(\phantom U, N)\big) \simeq \map^h_{\mathrm{mod}\text-\mathcal F_m}(\mathcal F_M, \widetilde{\mathcal F}_N).
\]

By~\hr[Theorem]{model-ind-res}, we need prove that there is a weak equivalence between $\widetilde\emb(\phantom U, N)$ (resp. $\emb(\phantom U, M)$) and $\widetilde{\mathcal F}_N$ (resp. $\mathcal F_M$) that carries the modules structure
from $L\mathbb D_m$ to $\mathcal F_m$.

First we note that by applying the strategy from the proof of Salvatore's zigzag of weak equivalences
(\hr[Proposition]{salv-w-eq}) we obtain a zigzag of right $W(L\mathbb D_m)$-modules
\[
	\emb(\phantom U, M)
	\xleftarrow{\sim}
	W\big(\emb(\phantom U, M)\big)
	\xrightarrow{\sim}
	\mathcal F_M.
\]

\begin{proposition}
There is a zigzag of homotopy equivalences
\begin{center}
\begin{tikzcd}
    \widetilde\emb(\phantom U, N)
    &
    W\big(\widetilde\emb(\phantom U, N)\big)
    \arrow[l, swap, "\sim"]
    \arrow[r, "\sim"]
    &
    \widetilde{\mathcal{F}}_N
    &
    \mathcal{F}_N
    \arrow[l, swap, "\sim"]
    \\
    L\mathbb D_m
    \arrow[u, symbol=\curvearrowright]
    &
    W(L\mathbb D_m)
    \arrow[u, symbol=\curvearrowright]
    \arrow[l, swap, "\sim"]
    \arrow[r, "\sim"]
    &
    \mathcal{F}_m
    \arrow[u, symbol=\curvearrowright]
    \arrow[r, equal]
    &
    \mathcal{F}_m.
    \arrow[u, symbol=\curvearrowright]
\end{tikzcd}
\end{center}
The bottom row indicates the underlying operad for a module. And the homotopy equivalences respect the restricted module structure.
\end{proposition}

\begin{corollary}
There is a weak equivalence of mapping spaces
\[
\map_{\mathrm{mod}_{\leqslant k}\text-\mathcal F_m}^h \big(\mathcal F_M, \widetilde{\mathcal F}_N\big)
\simeq
\map_{\mathrm{mod}_{\leqslant k}\text-\mathcal F_m}^h \big(\mathcal F_M, \mathcal F_N\big).
\]
\end{corollary}

\begin{proof}
To prove the equivalences we essentially just mimic Salvatore's argument.
The first arrow is given by sending a labeled tree to the corresponding composite of the labels in $\widetilde\emb(\phantom U, N)$. It is clearly a morphism of right $W(L \mathbb D_m)$-modules. And a homotopy equivalence is given by contracting edges.

To construct the second map note first that there exists an obvious morphism of symmetric sequences $r\colon\widetilde\emb(\phantom U, N) \to \widetilde{\mathcal F}_N$. It is defined by sending $\big((f_1, h_1), \ldots, (f_k, h_k)\big) \in \widetilde\emb(\phantom U, N)(k)$ to the restrictions to~$0$, i.e.
\[
    r_k\big((f_1, h_1), \ldots, (f_k, h_k)\big)
    {}={}
    \big((f_1(0), h_1|_{\{0\}\times I}), \ldots, (f_k(0), h_k|_{\{0\}\times I})\big) \in \widetilde{\mathcal F}_N(k).
\]
Since discs are contractible,  $r$ is an arity-wise homotopy equivalence. Thus, it is enough to extend $r$
to a morphism $R\colon W\big(\widetilde\emb(\phantom U, N)\big) \to \widetilde{\mathcal F}_N$ that is compatible with the module structures. Let $\tau$ be a representative of an element of $W\big(\widetilde\emb(\phantom U, N)\big)(k)$, i.e. $\tau$ is a (not necessarily two-levelled) tree with the root labeled by $\widetilde\emb(\phantom U, N)$ and other internal vertices labeled by $L\mathbb D_m$ with edges having length in $[0, 1]$. Suppose in addition that the lengths are in $(0,1)$. Let $m_t\colon D^m \to D^m$ be the dilation by $t$. From $\tau$ we construct a tree $\tau'$ with vertices labeled by $\widetilde\emb(\phantom U, N)$ and $L\mathbb D_m$. Combinatorially $\tau'$ is the same tree. For a vertex $v\in \tau$ decorated by embeddings $(f_1, \ldots, f_{|v|})$ and incoming edges of lengths $t_1, \ldots, t_k$ respectively, the corresponding vertex of $\tau'$ is decorated by rescaled embeddings
$(f_1\circ m_{1-t_1}, \ldots, f_{|v|}\circ m_{1-t_{|v|}})$. The path factor of the root label remains untouched. Now, define $F \in \widetilde\emb(\phantom U, N)(k)$ to be the composite of the labels of $\tau'$. And finally, $R_k(\tau) \coloneqq r_k(F)\in \widetilde{\mathcal F}_N(k)$. The map $R_k$ extends to $W\big(\widetilde\emb(\phantom U, N)\big)(k)$ by taking limits. Note that if $t_i \to 1$ the resulting tree is given by an operad action and the corresponding image goes to the strata. Thus, $R_k$'s define a right $W(L\mathbb D_m)$-module morphism.

Finally, the morphism $\mathcal F_N \to \widetilde{\mathcal F}_N$ sends the configuration to the same configuration equipped with the stationary path. It is a right $\mathcal F_m$-module map by the very definition. Since paths are contractible, it is a homotopy equivalence.
\end{proof}

\vspace{0.5em}

We proceed with the proof of \hr[Theorem]{main-convergence}. Since with respect to the model structure from~\hr[Section]{model-modules} every module is fibrant, it is enough to pass to the \emph{hairy} cofibrant resolution $\mathcal F^h_M$ of $\mathcal F_M$ from~\cite[p.~1252]{turc13} to construct the remaining horizontal arrows in~\eqref{theorem32-diagram}.

Finally, we need to describe the fiber of the right fibration.

\begin{lemma}
The fibre $F'$ of
$\map_{\mathrm{mod}\text-\mathcal F_m}(\mathcal F^h_M, \widetilde{\mathcal F}_N) \twoheadrightarrow
\map_{\mathrm{mod}\text-\mathcal F_m}(\mathcal F^h_M, \mathcal F^{m\text{-fr}}_N)$
over the image of an embedding $f \in \emb(M, N)$ is homotopy equivalent to
$\Gamma\Big(M, \Path^{fib}_{\sigma_{std}, df}\big(f^*\Frm(N)\big)\Big)$.
\end{lemma}
\begin{proof}
The fibre of $\map_{\mathrm{mod}\text-\mathcal F_m}(\mathcal F^h_M, \widetilde{\mathcal F}_N) \twoheadrightarrow
\map_{\mathrm{mod}\text-\mathcal F_m}(\mathcal F^h_M, \mathcal F^{m\text{-fr}}_N)$
over a given morphism $\mathcal F_M \to \mathcal F^{m\text{-fr}}_N$ coincides with the space of lifts
\[
\begin{tikzcd}
    &
    \widetilde{\mathcal F}_N
    \arrow[d, two heads]
    \\
    \mathcal F^h_M
    \arrow[r]
    \arrow[ur, dashed]
    &
    \mathcal F^{m\text{-fr}}_N,
\end{tikzcd}
\]
of the morphism of right $\mathcal F_m$-modules. By definition, this is the same as set of lifts
\begin{equation}\label{lifts-arity-r}
\begin{tikzcd}
    &
    \widetilde{\mathcal F}_N(r)
    \arrow[d, two heads]
    \\
    \mathcal F^h_M(r)
    \arrow[r]
    \arrow[ur, dashed]
    &
    \mathcal F^{m\text{-fr}}_N(r)
\end{tikzcd}
\end{equation}
that are compatible with the right $\mathcal F_m$-operadic action.

Note that an embedding $f \in \emb(M, N)$ defines a morphism $\mathcal F^h_M \to \mathcal F^{m\text{-fr}}_N$. Namely, the embedding fixes $m$-frames at configuration points $f(m_1),\ldots, f(m_r)$.

Since the horizontal arrow in~\eqref{lifts-arity-r} fixes the configuration (and frames at the configuration points), the lift $\mathcal F^h_M(r) \dashrightarrow \widetilde{\mathcal F}_N(r)$ is uniquely defined by the path factor, i.e. by a deformation of the standard $m$-frame $\sigma_{std}$ to the $m$-frame $df$ defined by the embedding $f\colon M \to N$ at the configuration points. The deformation of the $m$-frame at the configuration is a map
\begin{equation}\label{deformation-map}
    \mathcal F^h_M(r)\to\Big(\Path^{fib}_{\sigma_{std}, df}\big(f^*\Frm(N)\big)\Big)^{\times r}.
\end{equation}
over $M^{\times r} = \big(\mathcal F^h_M(1)\big)^{\times r}$, where $\mathcal F^h_M(r) \to \big(\mathcal F^h_M(1)\big)^{\times r}$ is given by a product of $0$-arity operadic actions. The map~\eqref{deformation-map} to the product is uniquely defined by projections onto the factors.
\[
\begin{tikzcd}
\mathcal{F}^h_M(r)
\arrow[r]
\arrow[dr]
&
\Big(\Path^{fib}_{\sigma_{std}, df}\big(f^*\Frm(N)\big)\Big)^{\times r}
\arrow[r, "\mathrm{Pr}_j"]
\arrow[d]
&
\Path^{fib}_{\sigma_{std}, df}\big(f^*\Frm(N)\big)
\arrow[d]
\\
&
\big(\mathcal F^h_M(1)\big)^{\times r} = M^{\times r}
\arrow[r, "\mathrm{Pr}_j"]
&
M
\end{tikzcd}
\]
The top line composite factors through
\[
\mathcal F^h_M(r) \to \mathcal F^h_M(1) = M \to \Path^{fib}_{\sigma_{std}, df}\big(f^*\Frm(N)\big).
\]
Thus, the space of lifts is equal to the space of maps $M \to \Path^{fib}_{\sigma_{std}, df}\big(f^*\Frm(N)\big)$ over $M$, i.e.
\[
\Gamma\Big(M, \Path^{fib}_{\sigma_{std}, df}\big(f^*\Frm(N)\big)\Big).
\]
\end{proof}

Finally, identifying the fibers $F$, $F'$ with $\Gamma\Big(M, \Path^{fib}_{\sigma_{std}, df}\big(f^*\Frm(N)\big)\Big)$ it is clear that the map $F \to F'$ induced by the morphism of fibrations
\[
\begin{tikzcd}
\widetilde\emb(M, N)
\arrow[d, two heads]
\arrow[r]
&
\map_{\mathrm{mod}\text-\mathcal F_m}(\mathcal F^h_M, \widetilde{\mathcal F}_N)
\arrow[d, two heads]
\\
\emb(M, N)
\arrow[r]
&
\map_{\mathrm{mod}\text-\mathcal F_m}(\mathcal F^h_M, \mathcal F^{m\text{-fr}}_N)
\end{tikzcd}
\]
is the identity, which concludes the proof of \hr[Theorem]{main-convergence}.
\qed

\section{Passing to graph complexes. Proof of \hr[Theorem]{main-computation}}

\subsection{Motivation} Let $\mathcal F^{m\text{-fr}}_N \to (\mathcal F^{m\text{-fr}}_N)^{\Q}$ be the rationalization morphism. It is expected that under certain conditions on the manifolds the canonical morphism
\[
R\colon\map^h_{\mathrm{mod}\text-\mathcal F_m}(\mathcal F_M, \mathcal F^{m\text{-fr}}_N)
\to
\map^h_{\mathrm{mod}\text-\mathcal F_m}(\mathcal F_M, (\mathcal F^{m\text{-fr}}_N)^{\mathbb Q})
\]
is a component-wise rational weak equivalence (see alsox~\cite[Theorem~1.2]{f-t-w20}). Therefore, description of the latter gives (potentially) the description of the rational homotopy type of
$\map^h_{\mathrm{mod}\text-\mathcal F_m}(\mathcal F_M, \mathcal F^{m\text{-fr}}_N)$,
and consequentely of the embedding space $\widetilde{\emb}(M, N)$.

\subsection{Proof of \hr[Theorem]{main-computation}}

\pasttheorem{main-computation}

\begin{proof}
We start from passing to the algebraic world via Quillen adjunction (see~\cite{will24})
\begin{align*}
	\map^h_{\mathrm{mod}\text-\mathcal F_m} (\mathcal F_M, \mathcal F_N^{\mathbb Q})
	\coloneqq
	\map^h_{\mathrm{mod}\text-\mathcal F_m} (\mathcal F_M, LG_{\sbullet}R\Omega_{\#}\mathcal F_N)
	\simeq
	\map^h_{\mathrm{dgHopf}\Omega_{\#}(\mathcal F_m)\text-\mathrm{comod}}(R\Omega_{\#}\mathcal F_N, R\Omega_{\#}\mathcal F_M).
\end{align*}

Since there is a weak equivalence $\Omega_{\#}(\mathcal F_m) \simeq e_m^c$ (see~\cite{fr-wi20}), we have an equivalence of the corresponding comodule categories (see~\cite[Theorem A.5]{will24}). In  particular, the latter mapping space is equivalent to
\[
	\map^h_{\mathrm{dgHopf}\Omega_{\#}(\mathcal F_m)\text-\mathrm{comod}}(R\Omega_{\#}\mathcal F_N, R\Omega_{\#}\mathcal F_M)
	\simeq
	\map^h_{\mathrm{dgHopf}e_m^c\text-\mathrm{comod}}(B_N, B_M),
\]
where $B_N$ and $B_M$ are $e_m^c$-comodules corresponding to $R\Omega_{\#}\mathcal F_N$ and $R\Omega_{\#}\mathcal F_M$, respectively.
Recall that $\cores^{e_m^c}_{e_n^c}(\Graphs^Z_{H^{\sbullet}, n})$ defines a cofibrant resolution for $B_N$ in the category of dg Hopf $e_m^c$-comodules. As for the target, we do not need a specific rational model, so denote by $\widehat R_M$ a fibrant rational model for $\mathcal F_M$. Thus, we get
\begin{align*}
	\map^h_{\mathrm{dgHopf}e_m^c\text-\mathrm{comod}}%
	(R\Omega_{\#}\mathcal F_N, R\Omega_{\#}\mathcal F_M)
	\coloneqq
	\map_{\mathrm{dgHopf}e_m^c\text-\mathrm{comod}}%
	(\cores^{e_m^c}_{e_n^c}(\Graphs^Z_{H^{\sbullet}, n}), \widehat{R}_M).
\end{align*}

In our codimension range $n - m \geqslant 2$ the canonical morphism $e_n^c \to e_m^c$ factors through $\Com^c$ (see~\cite{fr-wi20}). Therefore, we can factorise $\cores$ above as
\[
\cores^{e_m^c}_{e_n^c} = \cores^{e_m^c}_{\Com^c}\circ \cores^{\Com^c}_{e_n^c}.
\]
Using $\cores$-$\coind$-adjunction (see~\cite[Proposition~3.13]{will23}) we get a weak equivalence
\begin{align*}
	\map_{\mathrm{dgHopf}e_m^c\text-\mathrm{comod}}&%
	(\cores^{e_m^c}_{e_n^c}(\Graphs^Z_{H^{\sbullet}, n}), \widehat{R}_M)
	\\
	&=
	\map_{\mathrm{dgHopf}e_m^c\text-\mathrm{comod}}%
	(\cores^{e_m^c}_{\Com^c}\circ\cores^{\Com^c}_{e_n^c}(\Graphs^Z_{H^{\sbullet}, n}), \widehat{R}_M)
	\\
	&\simeq
	\map_{\mathrm{dgHopf}\Com^c\text-\mathrm{comod}}%
	\big(\cores^{\Com^c}_{e_n^c}(\Graphs^Z_{H^{\sbullet}, n}), \coind^{\Com^c}_{e_m^c}(\widehat R_M)\big)
	\\
	&\simeq
	\map_{\mathrm{dgHopf}\Com^c\text-\mathrm{comod}}%
	(\cores^{\Com^c}_{e_n^c}(\Graphs^Z_{H^{\sbullet}, n}), \mathbb F_{A_M}),
	\end{align*}
where $A_M$ is a Poincaré duality rational model for $M$ (see~\cite{la-st08}). The last weak equivalence is due to the fact the $\widehat R_M$ is of configuration space type (see~\cite{will23}).

The quasi-freeness of $\cores^{\Com^c}_{e_n^c}(\Graphs^Z_{H^{\sbullet}, n})$ as a dg Hopf $\Com^c$-comodule implies the following proposition.

\begin{proposition}[{\cite[Proposition~9.1]{will23}}]\label{hgc-bijection}
	There is a bijection
	\begin{align*}
		\varphi\colon
		\mor_{\mathrm{gHopf}\Com^c\text-\mathrm{comod}/\Omega^*(\Delta^{\sbullet})}
		\big(\cores^{\Com^c}_{e_n^c}(\Graphs^Z_{H^{\sbullet}, n})\otimes\Omega^*(\Delta^{\sbullet}), \mathbb F_{A_M}\otimes\Omega^*(\Delta^{\sbullet})\big)
		\\
		\to
		\mor_{g\mathbb S seq/\Omega^*(\Delta^{\sbullet})}
		\big(\pIG^Z_{H^{\sbullet}, n}\otimes\Omega^*(\Delta^{\sbullet}), \mathbb F_{A_M}\otimes\Omega^*(\Delta^{\sbullet})\big)
	\end{align*}
	that sends a morphism $F$ on the left-hand side to the composition with the inclusion of generators
	\[
	\pIG^Z_{H^{\sbullet}, n} \hookrightarrow \Graphs^Z_{H^{\sbullet}, n} \to \mathbb F_{A_M}.
	\]
\end{proposition}

Note that the proposition above only deals with graded Hopf $\Com^c$-comodule morphisms. To get an actual dg Hopf $\Com^c$-comodule morphism we need it in addition to commute with differetials. The latter leads us to the Maurer-Cartan space.

\begin{proposition}[{\cite[Corollary~9.2]{will23}}]\label{hgc-l-infty}
There is a filtered $L_\infty$-structure on $\hgc_{A_M, H^{\sbullet}, n}^Z$ such that
\[
	\map_{\mathrm{dgHopf}e_m^c\text-\mathrm{comod}}
	(\cores^{e_m^c}_{e_n^c}(\Graphs^Z_{H^{\sbullet}, n}), \widehat{R}_M)
	\cong
	\MC_{\sbullet}(\hgc_{A_M, H^{\sbullet}, n}^Z).
\]
\end{proposition}
\end{proof}

\subsection{Digression: recollections on $L_\infty$-algebras}
In this section we remind the construction of the generating function for $L_\infty$-algebras (see~\cite[Section~4.1]{fr-wi20-1}).

Let $L$ be a complete filtered $L_\infty$ algebra with the structure operations
\begin{equation}\label{l-infty-structure-operations}
	l_n\colon S^n\big(L[1]\big) \to L[1], n\geq 1.
\end{equation}
The complete filtration ensures the convergence of the series
\[
	\mathcal U(x) \coloneqq \sum_{n \geqslant 1} \frac{1}{n!} l_n(x, \ldots, x).
\]

Let $R$ be a graded commutative algebra. The complete tensor product $L \hat{\otimes} R$ is again an $L_\infty$ algebra equipped with a complete compatible filtration. Extending the coefficients $R$-linearly, we get the function
\[
	\mathcal U^R \colon (L\hat{\otimes} R)^1 \to (L\hat{\otimes} R)^2.
\]
The structure operations~\eqref{l-infty-structure-operations} can be recovered from $\mathcal U^R$ by graded polarization. Namely, for a collection $x_1, \ldots, x_n \in L$ of homogeneous elements, we consider the graded algebra $R = \Q[\varepsilon_1, \ldots, \varepsilon_n]$ generated by variables of degrees $|\varepsilon_i| = 1 - |x_i|$. Then $\pm l_n(x_1, \ldots, x_n)$ is the coefficient of the monomial $\varepsilon_1\cdots \varepsilon_n$ in $\mathcal U^R(x_1\varepsilon_1+\cdots+x_n \varepsilon_n)$.

Moreover, the structure relations are equivalent to the relation
\[
	\mathcal U^{R[\varepsilon]}\big(x + \varepsilon \mathcal U^R(x)\big) = \mathcal U^{R[\varepsilon]}(x)
\]
for the power series $\mathcal U^R$, for any graded commutative algebra $R$, any element $x \in (L\hat{\otimes} R)^1$, where $\varepsilon$ is a formal variable of degree $-1$.

\subsection{Combinatorial description for the $L_\infty$-structure on the hairy graph complex}

Here we give combinatorial description on $\hgc_{A_M, H^{\sbullet}, n}^Z$ from \hr[Proposition]{hgc-l-infty}.

Let $\Phi$ be the isomorphism inverse to $\varphi$ from \hr[Proposition]{hgc-bijection}
\[
	\Phi\colon
	\mor
	\big(\pIG^Z_{H^{\sbullet}, n}\otimes\Omega^*(\Delta^{\sbullet}), \mathbb F_{A_M}\otimes\Omega^*(\Delta^{\sbullet})\big)
	\xrightarrow{\cong} 
	\mor
	\big(\cores^{\Com^c}_{e_n^c}(\Graphs^Z_{H^{\sbullet}, n})\otimes\Omega^*(\Delta^{\sbullet}), \mathbb F_{A_M}\otimes\Omega^*(\Delta^{\sbullet})\big).
\]
Then the $L_\infty$-structure is defined by the generating function
\[
	\mathcal U^{\Omega^*(\Delta^{\sbullet})}\colon \big(\hgc_{A_M, H^{\sbullet}, n}\hat{\otimes}\Omega^*(\Delta^{\sbullet})\big)^1 \to \big(\hgc_{A_M, H^{\sbullet}, n}\hat{\otimes}\Omega^*(\Delta^{\sbullet})\big)^2
\]
defined by the formula
\[
	\mathcal U^{\Omega^*(\Delta^{\sbullet})}(x)
	\coloneqq
	\big[\mathrm d_{\widehat{R}_M}\circ \Phi(x) - \Phi(x)\circ \mathrm d_{\Graphs_{H^{\sbullet}, n}^Z} \big]\circ \iota,
\]
where $\iota\colon \pIG_{H^{\sbullet}, n}^Z\hookrightarrow \Graphs_{H^{\sbullet}, n}^Z$ is the canonical inclusion. 
Note that $\mathrm d_{\Graphs_{H^{\sbullet}, n}^Z} = \mathrm d_{\Graphs_{H^{\sbullet}, n}} + (Z\cdot)$. 
We can further decompose the differential $\mathrm d_{\Graphs_{H^{\sbullet}, n}^Z}$ with respect to the internally connected generators
\[
	\mathrm d_{\Graphs_{H^{\sbullet}, n}^Z}
	=
	\mathrm d_{int}
	+
	\sum_{k \geqslant 1} \mathrm d_{ext}^k
	+
	\sum_{k \geqslant 1} (Z\cdot)^k.
\]
Here $\mathrm d_{int}$ is the part of differential contracting internal edges, in particular, it leaves the graph internally connected. The summands $\mathrm d_{ext}^k$ (resp. $(Z\cdot)^k$) correspond to the part of the differential (resp. $(Z\cdot)$) that sends the generators $\IG_{H^{\sbullet}, n}$ to $S^k(\IG_{H^{\sbullet}, n})$ induced by contracting an edge between internal and external vertices (resp. "cutting off" a subgraph isomorphic to $Z$):
\begin{equation}\label{d-z-k}
\begin{alignedat}{3}
\mathrm d_{ext}^k&\colon
\begin{tikzpicture}[baseline=-.8ex]
	\node[draw,circle] (v) at (0,.3) {$\Gamma_1$};
	\node (v2) at (1.25,.3) {$\scriptstyle \cdots$};
	\node[draw,circle] (v1) at (2.5,.3) {$\Gamma_k$};
	\node[ext] (w1) at (-.7,-.5) {};
	\node[ext] (w2) at (-.25,-.5) {};
	\node[ext] (w4b) at (1.25,-.5) {};
	\node[int] (w4) at (1.25,-.2) {};
	\node[ext] (w6) at (2.75,-.5) {};
	\node[ext] (w7) at (3.2,-.5) {};
	\draw (v) edge (w1)  edge (w2) edge[bend left] (w4) edge[bend right] (w4) edge (w4)
	(v1) edge (w4) edge[bend left] (w4) edge (w6) edge (w7)
	(w4b) edge (w4)
	(v2) edge[bend left] (w4) edge (w4);
\end{tikzpicture}
&&\mapsto
&&\begin{tikzpicture}[baseline=-.8ex]
	\node[draw,circle] (v) at (0,.3) {$\Gamma_1$};
	\node (v2) at (1.25,.3) {$\scriptstyle \cdots$};
	\node[draw,circle] (v1) at (2.5,.3) {$\Gamma_k$};
	\node[ext] (w1) at (-.7,-.5) {};
	\node[ext] (w2) at (-.25,-.5) {};
	\node[ext] (w4b) at (1.25,-.5) {};
	\node[ext] (w6) at (2.75,-.5) {};
	\node[ext] (w7) at (3.2,-.5) {};
	\draw (v) edge (w1)  edge (w2) edge[bend left] (w4b) edge[bend right] (w4b) edge (w4b)
	(v1) edge (w4b) edge[bend left] (w4b) edge (w6) edge (w7)
	(v2) edge[bend left] (w4b) edge (w4b);
\end{tikzpicture};
\\
(Z\cdot)^k&\colon
\hspace{3.25em}
\begin{tikzpicture}[baseline=-.8ex]
	\node[draw,circle] (v) at (0,.3) {$\Gamma$};
	\node[ext] (w1) at (-.7,-.5) {};
	\node[ext] (w2) at (-.35,-.5) {};
	\node[label=center:{$\scriptstyle\cdots$}] (w3) at (0,-.5) {};
	\node[ext] (w4) at (.35,-.5) {};
	\node[ext] (w5) at (.7,-.5) {};
	\draw (v) edge[bend right] (w1) edge (w2) edge (w4) edge[bend left] (w5);
\end{tikzpicture}
&&\mapsto
\sum\pm
&&\begin{tikzpicture}[baseline=-.8ex]
	\node[draw,circle,dashed] (v) at (0,1.3) {$Z$};
	\node[draw,circle] (v1) at (-1.2,.3) {$\scriptstyle\Gamma_1$};
	\node[label=center:{$\cdots$}] (v2) at (0,.3) {};
	\node[draw,circle] (v3) at ( 1.2,.3) {$\scriptstyle\Gamma_k$};
	\node[ext] (w11) at (-1.9,-0.5) {};
	\node[ext] (w12) at (-1.55,-0.5) {};
	\node[label=center:{$\scriptstyle\cdots$}] (w13) at (-1.2,-0.5) {};
	\node[ext] (w14) at (-.85,-0.5) {};
	\node[ext] (w15) at (-.5,-0.5) {};
	\node[ext] (w21) at (.5,-0.5) {};
	\node[ext] (w22) at (.85,-0.5) {};
	\node[label=center:{$\scriptstyle\cdots$}] (w23) at (1.2,-0.5) {};
	\node[ext] (w24) at (1.55,-0.5) {};
	\node[ext] (w25) at (1.9,-0.5) {};
	\node[ext] (w1) at (2.4,-0.5) {};
	\node[label=center:{$\scriptstyle\cdots$}] (w2) at (2.75,-0.5) {};
	\node[ext] (w3) at (3.1,-0.5) {};
	\draw (v) edge[dashed,bend right] (v1)
			  edge[dashed] (v1)
			  edge[dashed,bend left] (v1)
			  edge[dashed,bend right] (v3)
			  edge[dashed] (v3)
			  edge[dashed,bend left] (v3)
			  edge[dashed,bend left] (w1)
			  edge[dashed,bend left] (w3);
	\draw (v1) edge[bend right] (w11)
			   edge (w12)
			   edge (w14)
			   edge[bend left] (w15);
	\draw (v3) edge[bend right] (w21)
			   edge (w22)
			   edge (w24)
			   edge[bend left] (w25);
	\draw[dashed,thick] (-1.7,.8) -- (2.75,.8);
\end{tikzpicture}.
\end{alignedat}
\end{equation}

In particular, the structure morphisms $l_k$, $k\geqslant 2$ have the following from
\[
	l_k = l_k^{std} + l_k^Z,
\]
where $l_k^{std}$ is the standard "untwisted" $L_\infty$-structure morphism defined by
\[
	l_k^{std} \coloneqq -\Phi(x)\circ \mathrm d_{ext}^k \circ \iota,
\]
and $l_k^Z$ is the part related to the twist by~$Z$:
\[
	l_k^{Z} \coloneqq -\Phi(x)\circ (Z\cdot)^k \circ \iota.
\]
Thus, dualizing~\eqref{d-z-k} we get our structure morphisms:

\begin{alignat}{3}
l^{std}_k&\Big(
\begin{tikzpicture}[baseline=-.8ex]
	\node[draw,circle] (v) at (0,.3) {$\Gamma_1$};
	\node[ext,label=270:{\scalebox{0.5}{$a^1_1$}}] (w1) at (-.7,-.5) {};
	\node[ext,label=270:{\scalebox{0.5}{$a^1_2$}}] (w2) at (-.35,-.5) {};
	\node[label=center:{$\scriptstyle\cdots$}] (w3) at (0,-.5) {};
	\node[ext,label=270:{\scalebox{0.5}{$a^1_{j_1-1}$}}] (w4) at (.35,-.5) {};
	\node[ext,label=270:{\scalebox{0.5}{$a^1_{j_1}$}}] (w5) at (.7,-.5) {};
	\draw (v) edge[bend right] (w1) edge (w2) edge (w4) edge[bend left] (w5);
\end{tikzpicture};
\cdots;
\begin{tikzpicture}[baseline=-.8ex]
	\node[draw,circle] (v) at (0,.3) {$\Gamma_k$};
	\node[ext,label=270:{\scalebox{0.5}{$a^k_1$}}] (w1) at (-.7,-.5) {};
	\node[ext,label=270:{\scalebox{0.5}{$a^k_2$}}] (w2) at (-.35,-.5) {};
	\node[label=center:{$\scriptstyle\cdots$}] (w3) at (0,-.5) {};
	\node[ext,label=270:{\scalebox{0.5}{$a^k_{j_1-1}$}}] (w4) at (.35,-.5) {};
	\node[ext,label=270:{\scalebox{0.5}{$a^k_{j_1}$}}] (w5) at (.7,-.5) {};
	\draw (v) edge[bend right] (w1) edge (w2) edge (w4) edge[bend left] (w5);
\end{tikzpicture}
\Big)
&&=
\sum
\pm
&&
\hspace{.35em}
\begin{tikzpicture}[baseline=-.8ex]
	\node[draw,circle] (v1) at (0,.3) {$\Gamma_1$};
	\node[draw,circle] (v2) at (2.5,.3) {$\Gamma_k$};
	\node[ext,label=270:{$\scriptstyle a$}] (w1) at (-.5,-.5) {};
	\node[ext,label=270:{$\scriptstyle a$}] (w2) at (.2,-.5) {};
	\node[ext,label=270:{\scalebox{0.5}{$\prod a$}}] (w3) at (1.25,-.5) {};
	\node[ext,label=270:{$\scriptstyle a$}] (w4) at (2.3,-.5) {};
	\node[ext,label=270:{$\scriptstyle a$}] (w5) at (3,-.5) {};
	\node[int] (w) at (1.25,-.2) {};
	\node at (1.25,.3) {$\scriptstyle \cdots$};
	\node at (-.15,-.5) {$\scriptstyle \cdots$};
	\node at (2.65,-.5) {$\scriptstyle \cdots$};
	\draw (v1) edge[bend left] (w) edge[bend right] (w) edge (w);
	\draw (v2) edge[bend right] (w) edge[bend left] (w);
	\draw (w) edge (w3);
	\draw (v1) edge[bend right] (w1) edge (w2);
	\draw (v2) edge (w4) edge[bend left] (w5);
\end{tikzpicture};
\label{lkstd}\\
l^Z_k&\Big(
\begin{tikzpicture}[baseline=-.8ex]
	\node[draw,circle] (v) at (0,.3) {$\Gamma_1$};
	\node[ext,label=270:{\scalebox{0.5}{$a^1_1$}}] (w1) at (-.7,-.5) {};
	\node[ext,label=270:{\scalebox{0.5}{$a^1_2$}}] (w2) at (-.35,-.5) {};
	\node[label=center:{$\scriptstyle\cdots$}] (w3) at (0,-.5) {};
	\node[ext,label=270:{\scalebox{0.5}{$a^1_{j_1-1}$}}] (w4) at (.35,-.5) {};
	\node[ext,label=270:{\scalebox{0.5}{$a^1_{j_1}$}}] (w5) at (.7,-.5) {};
	\draw (v) edge[bend right] (w1) edge (w2) edge (w4) edge[bend left] (w5);
\end{tikzpicture};
\cdots;
\begin{tikzpicture}[baseline=-.8ex]
	\node[draw,circle] (v) at (0,.3) {$\Gamma_k$};
	\node[ext,label=270:{\scalebox{0.5}{$a^k_1$}}] (w1) at (-.7,-.5) {};
	\node[ext,label=270:{\scalebox{0.5}{$a^k_2$}}] (w2) at (-.35,-.5) {};
	\node[label=center:{$\scriptstyle\cdots$}] (w3) at (0,-.5) {};
	\node[ext,label=270:{\scalebox{0.5}{$a^k_{j_1-1}$}}] (w4) at (.35,-.5) {};
	\node[ext,label=270:{\scalebox{0.5}{$a^k_{j_1}$}}] (w5) at (.7,-.5) {};
	\draw (v) edge[bend right] (w1) edge (w2) edge (w4) edge[bend left] (w5);
\end{tikzpicture}
\Big)
&&=
\sum
\pm
&&\begin{tikzpicture}[baseline=-.8ex]
	\node[draw,circle] (v) at (0,1.3) {$Z$};
	\node[draw,circle] (v1) at (-1.2,.3) {$\scriptstyle\Gamma_1$};
	\node[label=center:{$\cdots$}] (v2) at (0,.3) {};
	\node[draw,circle] (v3) at ( 1.2,.3) {$\scriptstyle\Gamma_k$};
	\node[ext,label=270:{\scalebox{0.5}{$a^1_1$}}] (w11) at (-1.9,-0.5) {};
	\node[ext,label=270:{\scalebox{0.5}{$a^1_2$}}] (w12) at (-1.55,-0.5) {};
	\node[label=center:{$\scriptstyle\cdots$}] (w13) at (-1.2,-0.5) {};
	\node[ext,label=270:{\scalebox{0.5}{$a^1_{j_1-1}$}}] (w14) at (-.85,-0.5) {};
	\node[ext,label=270:{\scalebox{0.5}{$a^1_{j_1}$}}] (w15) at (-.5,-0.5) {};
	\node[ext,label=270:{\scalebox{0.5}{$a^k_1$}}] (w21) at (.5,-0.5) {};
	\node[ext,label=270:{\scalebox{0.5}{$a^k_2$}}] (w22) at (.85,-0.5) {};
	\node[label=center:{$\scriptstyle\cdots$}] (w23) at (1.2,-0.5) {};
	\node[ext,label=270:{\scalebox{0.5}{$a^k_{j_k-1}$}}] (w24) at (1.55,-0.5) {};
	\node[ext,label=270:{\scalebox{0.5}{$a^k_{j_k}$}}] (w25) at (1.9,-0.5) {};
	\node[ext,label=270:{\scalebox{0.5}{$1$}}] (w1) at (2.4,-0.5) {};
	\node[label=center:{$\scriptstyle\cdots$}] (w2) at (2.75,-0.5) {};
	\node[ext,label=270:{\scalebox{0.5}{$1$}}] (w3) at (3.1,-0.5) {};
	\draw (v) edge[bend right] (v1)
	edge (v1)
	edge[bend left] (v1)
	edge[bend right] (v3)
	edge (v3)
	edge[bend left] (v3)
	edge[bend left] (w1)
	edge[bend left] (w3);
	\draw (v1) edge[bend right] (w11)
	edge (w12)
	edge (w14)
	edge[bend left] (w15);
	\draw (v3) edge[bend right] (w21)
	edge (w22)
	edge (w24)
	edge[bend left] (w25);
\end{tikzpicture}.
\label{lkz}
\end{alignat}
Where in~\eqref{lkstd} the sum runs over all possible non-empty subsets of hairs of $\Gamma_1, \ldots, \Gamma_k$, glues given subsets to a new internal vertex that is connected to a new external vertex labeled by the product of the corresponding hair labels, and the sum in~\eqref{lkz}
runs over all possible ways to attach hairs of $Z$ to $\Gamma_1, \ldots, \Gamma_k$ to obtain an internally connected graph, with hairs of $Z$ labeled by~$1$ allowed.

Finally, we describe the differential
\[
	\mathrm d =	\delta_{\widehat{R}_M} + \delta_{split} + \delta_{join} + \delta_{Z} \coloneqq \big[\mathrm d_{\widehat{R}_M}\circ \Phi(x) - \Phi(x)\circ (\mathrm d_{int} + \mathrm d_{ext}^1 + (Z\cdot)^1)\big]\circ \iota.
\]
The first part is induced by the inner $\widehat{R}_M$ differential. The second part is induced by splitting internal vertices in all possible ways (and distributing the labels). The last two are similar to~\eqref{lkstd} and \eqref{lkz}:
\[
\delta_{join}%
\begin{tikzpicture}[baseline=-.8ex]
	\node[draw,circle] (v) at (0,.3) {$\Gamma$};
	\node[ext,label=270:{\scalebox{0.5}{$a_1$}}] (w1) at (-.7,-.5) {};
	\node[ext,label=270:{\scalebox{0.5}{$a_2$}}] (w2) at (-.35,-.5) {};
	\node[ext,label=270:{\scalebox{0.5}{$a_{k-1}$}}] (w4) at (.35,-.5) {};
	\node[ext,label=270:{\scalebox{0.5}{$a_k$}}] (w5) at (.7,-.5) {};
	\node[label=center:{$\scriptstyle\cdots$}] at (0,-.5) {};
	\draw (v) edge[bend right] (w1) edge (w2) edge (w4) edge[bend left] (w5);
\end{tikzpicture}
=
\sum_{\substack{S \subseteq \text{Hairs} \\ |S|\geqslant 2}}
\begin{tikzpicture}[baseline=-.8ex]
	\node[draw,circle] (v) at (0,.3) {$\Gamma$};
	\node[ext,label=270:{\scalebox{0.7}{$a_{i_1}$}}] (w1) at (-.7,-.5) {};
	\node[ext,label=270:{\scalebox{0.7}{$a_{i_2}$}}] (w2) at (-.35,-.5) {};
	\node[ext,label=270:{\scalebox{0.7}{$\prod_{j\in S} a_j$}}] (w3) at (.53,-.5) {};
	\node[label=center:{$\scriptstyle\cdots$}] at (0,-.5) {};
	\node[int] (w) at (.53,-.2) {};
	\draw (v) edge[bend right] (w1)
			  edge (w2)
			  edge (w)
			  edge[bend left] (w)
			  edge[bend right] (w);
	\draw (w) edge (w3);
\end{tikzpicture};
\qquad
\delta_{Z}%
\begin{tikzpicture}[baseline=-.8ex]
	\node[draw,circle] (v) at (0,.3) {$\Gamma$};
	\node[ext,label=270:{\scalebox{0.5}{$a_1$}}] (w1) at (-.7,-.5) {};
	\node[ext,label=270:{\scalebox{0.5}{$a_2$}}] (w2) at (-.35,-.5) {};
	\node[ext,label=270:{\scalebox{0.5}{$a_{k-1}$}}] (w4) at (.35,-.5) {};
	\node[ext,label=270:{\scalebox{0.5}{$a_k$}}] (w5) at (.7,-.5) {};
	\node[label=center:{$\scriptstyle\cdots$}] at (0,-.5) {};
	\draw (v) edge[bend right] (w1) edge (w2) edge (w4) edge[bend left] (w5);
\end{tikzpicture}
=
\sum
\pm
\begin{tikzpicture}[baseline=-.8ex]
	\node[draw,circle] (v) at (0,1.3) {$Z$};
	\node[draw,circle] (v1) at (0,.3) {$\scriptstyle\Gamma$};
	\node[ext,label=270:{\scalebox{0.5}{$a_1$}}] (w11) at (-.7,-0.5) {};
	\node[ext,label=270:{\scalebox{0.5}{$a_2$}}] (w12) at (-.35,-0.5) {};
	\node[label=center:{$\scriptstyle\cdots$}] (w13) at (0,-0.5) {};
	\node[ext,label=270:{\scalebox{0.5}{$a_{k-1}$}}] (w14) at (.35,-0.5) {};
	\node[ext,label=270:{\scalebox{0.5}{$a_{k}$}}] (w15) at (.7,-0.5) {};
	\node[ext,label=270:{\scalebox{0.5}{$1$}}] (w1) at (1.25,-0.5) {};
	\node[label=center:{$\scriptstyle\cdots$}] (w2) at (1.6,-0.5) {};
	\node[ext,label=270:{\scalebox{0.5}{$1$}}] (w3) at (1.95,-0.5) {};
	\draw (v) edge[bend right] (v1)
	edge (v1)
	edge[bend left] (v1)
	edge[bend left] (w1)
	edge[bend left] (w3);
	\draw (v1) edge[bend right] (w11)
	edge (w12)
	edge (w14)
	edge[bend left] (w15);
\end{tikzpicture},
\]

where the latter sum again runs over all ways to attach hairs of $Z$ to $\Gamma$.

\section{Applications}

In this section we give some examples of computations of the right-hand side in \hr[Theorem]{main-computation}. Despite the fact that in the examples below spaces are not parallelizable, it was shown in \cite{ca-wi23} that the rationaliation of the corresponding Fulton-MacPherson completions have a $\mathcal F_m^\Q$-module structure, which, in turn, makes the mapping space well defined. In particular, despite the fact that $\widetilde{\emb}$ is not defined, the calculations provide same amount of information.

\subsection{Comparison: embeddings into $S^n$}
Let $Z \coloneqq
2
\begin{tikzpicture}[baseline=-.65ex]
	\node (v) at (0,0) {$\omega$};
	\node (w) at (1.5,0) {$1$};
	\draw (v) edge[bend left] (w);
\end{tikzpicture}
\in \hgc_{\overline H^{\sbullet}(S^n), n}$  be a Maurer-Cartan element. The differential in the twisted hairy graph complex $\hgc_{\overline H^{\sbullet}(S^n), H^{\sbullet}(S^k), n}$ is split into three pieces:
\[
	\delta = \delta_{split} + \delta_{join} + (Z\cdot),
\]
where $(Z\cdot)$ can itself be split into two pieces: $(Z\cdot) = (Z\cdot)^{hair} + (Z\cdot)^{edge}$. The piece $(Z\cdot)^{hair}$ swaps an $\omega$-decoration of an internal vertex to a hair decorated by~$1$, and $(Z\cdot)^{edge}$ adds an internal edge starting at $\omega$-decoration. 
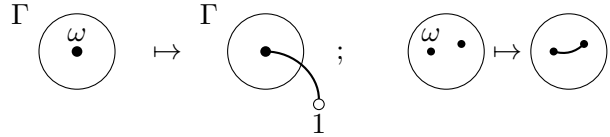
\[
\begin{tikzpicture}
	\draw (0,0) circle (.5);
	
	\fill (0,0) circle (2pt);
	
	\node[above] at (0,0) {$\omega$};
	
	\node[right] at (-1,.5) {$\Gamma$};
\end{tikzpicture}
\quad
\mapsto
\begin{tikzpicture}
	\draw (0,0) circle (.5);
	\fill (0,0) circle (2pt);
	\node at (0,0) {};
	\node[right] at (-1,.5) {$\Gamma$};
	\draw[thick, shorten >= 2pt] (0,0) to[out=0, in=90] (.7, -.7);
	\draw (.7,-.7) circle (2pt);
	\node[below] at (.7, -.7) {$1$};
\end{tikzpicture};
\qquad
\begin{tikzpicture}
	\draw (0,0) circle (0.5);
	
	\fill (-0.2,0) circle (1.5pt); 
	\fill (0.2,0.1) circle (1.5pt);  
	
	\node[above] at (-0.2, 0) {$\omega$};
	
\end{tikzpicture}
\mapsto
\begin{tikzpicture}
	\draw (0,0) circle (0.5);
	
	\fill (-0.2,0) circle (1.5pt); 
	\fill (0.2,0.1) circle (1.5pt);  
	
	\node at (-0.2, 0) {};
	\draw[thick] (-0.2,0) to[bend right] (0.2,0.1);
\end{tikzpicture}
\]

Let $i \colon \hgc_{\overline H^{\sbullet}(S^k), n} \hookrightarrow \hgc_{\overline H^{\sbullet}(S^n), H^{\sbullet}(S^k), n}$ be the inclusion of a subcomplex, and the $\cone(i)$ be the cone.
Define a filtration $F_{\sbullet}\cone(i)$ by the number of internal edges on the cone:
\[
    F_p\cone(i) \coloneqq \big\{(\Gamma, \Gamma')\in \hgc_{\overline H^{\sbullet}(S^k), n}[1]\oplus \hgc_{\overline H^{\sbullet}(S^n), H^{\sbullet}(S^k), n}\,|\; |E^i(\Gamma)| = |E^i(\Gamma')| \leq p\big\}.
\]
The associated graded $Gr_p\cone(i)$ has the differential
\[
    (\Gamma, \Gamma') \xmapsto{\partial} \big(0, \Gamma + (Z\cdot)^{hair}(\Gamma')\big).
\]
Therefore, the first page has form
\[
    E^1_{\sbullet, \sbullet} = H_{\sbullet}(\cone(i); \partial)
    =
    H_{\sbullet}(\hgc_{\overline H^{\sbullet}(S^n), H^{\sbullet}(S^k), n}/\hgc_{\overline H^{\sbullet}(S^k), n}; (Z\cdot)^{hair}).
\]
The latter vector space consists of graphs without hairs and with at least one vertex decorated by~$\omega$. Indeed, each graph with a hair can be obtained as the image under $(Z\cdot)^{hair}$ and the graphs without decorations by $\omega$ belong to $\hgc_{\overline H^{\sbullet}(S^k), n}$.
Denote the latter vector space by $V$. The second page then has form
\[
    E^2_{\sbullet, \sbullet} = H_{\sbullet}(V; \delta_{split} + (Z\cdot)^{edge}).
\]
The space~$V$ can be identified with the space of undecorated hairy graphs, where $(Z\cdot)^{edge}$ acts 
by attaching one of the hairs to an internal vertex different from the initial vertex.

By \cite[Theorem~1.1]{zivk19} (see also~\cite[Theorem~1(ii)]{f-n-w23}), the cohomology of the complex above is trivial. Therefore, the spectral sequence degenerates at $E^2_{\sbullet, \sbullet}$.
Thus, the inclusion
\[
	i \colon \hgc_{\overline H^{\sbullet}(S^k), n} \hookrightarrow \hgc_{\overline H^{\sbullet}(S^n), H^{\sbullet}(S^k), n}
\]
is a quasi-isomorphism.

By \cite{ar-tu14}, the source computes rational homotopy groups of the space $\overline\emb_{\partial}(\R^k, \R^n)$ of long embeddings $\R^k \hookrightarrow \R^n$. Thus, the computation above shows that the spaces $\overline\emb_{\partial}(\R^k, \R^n)$ and $\widetilde\emb(S^k, S^n)$ are rationally equivalent.

\begin{remark}
The computation above works for any general source, i.e. the inclusion
\[
    i \colon \hgc_{\overline H^{\sbullet}(M), n} \hookrightarrow \hgc_{\overline H^{\sbullet}(S^n), H^{\sbullet}(M), n}
\]
is a quasi-isomorphism.
\end{remark}

\subsection{Embeddings into $S^d \times S^d \setminus\!\{pt\}$}
The embedding space $\emb(S^k, S^d\times S^d \setminus\!\{pt\})$ has naturally three distinguished points given by the factor embeddings $S^d \hookrightarrow S^d\times S^d\setminus\{pt\}$ and the Haefliger embedding $S^{4n-1} \hookrightarrow S^{3n}\times S^{3n}\setminus\!\{pt\}$. In the following we describe three distinguished elements in $\MC_1(\hgc_{H^{\sbullet}(S^k), \overline H^{\sbullet}(S^d\times S^d\setminus\!\{pt\}), 2d})$ that are expected to correspond to the embeddings above.

Let $Z \coloneqq
\begin{tikzpicture}[baseline=-.65ex]
	\node (v) at (0,0) {$\omega_1$};
	\node (w) at (1.5,0) {$\omega_2$};
	\draw (v) edge[bend left] (w);
\end{tikzpicture}
\in \MC(\hgc_{\overline H^{\sbullet}(S^d\times S^d\setminus\!\{pt\}), 2d})$ be a Maurer-Cartan element. Here $\omega_1$, $\omega_2$ are generators of $H^*(S^d \times S^d \setminus \{pt\}; \mathbb Q) \cong \mathbb Q\langle \omega_1, \omega_2 \rangle$. Our goal is to describe the degree one part of the Maurer-Cartan set $\MC_1(\hgc_{A_M, \overline H^{\sbullet}(S^d\times S^d\setminus\!\{pt\}), 2d})$, where $A_M$ is a rational model for the source $M$ of the embedding $M \hookrightarrow S^d\times S^d \setminus \{pt\}$.

Recall that for the graphs from $\hgc_{A_M, \overline H^{\sbullet}(S^d\times S^d\setminus\!\{pt\}), 2d}$ the degree is given by
\[
	(2d-1)e - 2dv - (\text{degrees of hair decorations}) + (\text{degrees of internal vertex decorations}),
\]
where $e$ and $v$ denote the number of edges and internal vertices respectively.
In particular, with the shift we get extra $+1$:
\[
(2d-1)e - 2dv - (\text{degrees of hair decorations}) + (\text{degrees of internal vertex decorations}) + 1.
\]

As in \cite{f-t-w17} we have that only trees contribute to the degree~$1$ component. We claim that only there no graphs with hairs decorated by $1$ contribute to degree one. The minimal degree of a graph in $\hgc_{A_M, \overline H^{\sbullet}(S^d\times S^d\setminus\!\{pt\}), 2d}[1]$ is
\[
    (2d-1)e - 2dv - \dim(A)h + 1 = - (2d-3) + (2e-3v) + (2d - \dim(A) - 3) h + 1 > - (2d-3) + 1,
\]
where $e$, $v$ and $h$ denote the number of edges, internal vertices and hairs respectively. Therefore, if there is at least one hair decorated by $1 \in A_M$ the degree will be
\[
    (2d-1)e - 2dv - \dim(A)(h-1) + 1 > -(2d-3) + \dim(A) + 1 > 1,
\]
as we consider only codimension at least~$3$. Thus, degree zero graphs have no $1 \in A_M$ hair decorations.

Finally, we show that there is only one internal vertex decoration. Graphs with minimal degrees are unitrivalent trees. Such graphs have degree
\begin{align*}
    (2d-1)(2v+1) - 2dv - \dim(A)(v+2) + 1
    &= (2d-1)(v-1) - 2dv + (2d-\dim(A)-1)(v+2) + 1
    \\
    &= (2d-\dim(A)-2)v + (2d-2\dim(A)-1) + 1.
\end{align*}
Adding an internal vertex decoration increases degree by at least $\dim(A) - d + 1$. The minimal case happens if we remove a hair and add an internal vertex decoration. Therefore, graphs with two (and more) internal vertex decorations have degree at least:
\[
    (2d-\dim(A)-2)v + 2d - 2\dim(A) - 1 + 2(\dim(A)-d+1) + 1 = (2d-\dim(A)-2)v + 2 > 2.
\]
Thus, such graphs do not contribute to the one degree part of the Maurer-Cartan space.

Applying IHX relations we remain with very few underlying graphs:
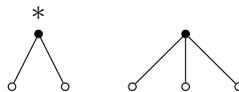
\[
\begin{tikzpicture}
	\node[ext] (v1) at (0,0) {};
	\node[ext] (v2) at (.7,0) {};
	\node[int,label=90:$*$] (i1) at (.35,.7) {};
	\draw (i1) edge (v1)
	(i1) edge (v2);
\end{tikzpicture}
\qquad
\begin{tikzpicture}
	\node[ext] (v1) at (0,0) {};
	\node[ext] (v2) at (.7,0) {};
    \node[ext] (v3) at (1.4,0) {};
	\node[int] (i1) at (.7,.7) {};
	\draw (i1) edge (v1) edge (v2) edge (v3);
\end{tikzpicture}
\]

We are interested in the case $M = S^k$. In this formal case $A_M$ can be taken to be the cohomology ring $H^{\sbullet}(S^k) \cong \Q[\omega]/(\omega^2)$ and the only graphs respecting the constraints above are
\[
\begin{tikzpicture}
	\node (v1) at (0,0) {$\scriptstyle\omega$};
	\node (v2) at (.7,0) {$\scriptstyle\omega$};
	\node[int, label=90:{$\scriptstyle\omega_i$}] (i1) at (.35,.7) {};
	\draw (i1) edge (v1) 
	(i1) edge (v2);
\end{tikzpicture},
i = 1,2;
\qquad
\begin{tikzpicture}
	\node (v1) at (0,0) {$\scriptstyle\omega$};
	\node (v2) at (.7,0) {$\scriptstyle\omega$};
    \node (v3) at (1.4,0) {$\scriptstyle\omega$};
	\node[int] (i1) at (.7,.7) {};
	\draw (i1) edge (v1) edge (v2) edge (v3);
\end{tikzpicture}.
\]

\end{document}